\documentclass[12pt]{article}
\usepackage[french, english]{babel}
\usepackage[utf8]{inputenc}
\usepackage{amsmath,amsfonts,amssymb,amsthm,mathrsfs,bbm}
\usepackage{minitoc}
\usepackage[normalem]{ulem}

\renewcommand{\leq}{\leqslant}
\renewcommand{\geq}{\geqslant}

\newcommand{\rev}[1]{#1}

\usepackage{wasysym}

\usepackage[dvipsnames]{xcolor}
\usepackage[a4paper,vmargin={2cm,2cm},hmargin={1cm,1cm}]{geometry}
\usepackage{graphicx,graphics}
\usepackage{epsfig}
\usepackage{latexsym}
\usepackage{stmaryrd}
\usepackage{ae,aecompl}

\usepackage[english]{babel}
 \usepackage[colorlinks=true]{hyperref}
\usepackage{pstricks}
\usepackage{enumerate}
\usepackage{tikz}
\usetikzlibrary{arrows,automata}

\renewcommand{\P}{\mathbb{P}}
\newcommand{\E}{\mathbb{E}}
\newcommand{\ext}{\mathrm{ext}}
\newcommand{\epsb}{ \boldsymbol{\varepsilon}}
\newcommand{\xb}{\boldsymbol{t}}

\newtheorem{theorem}{Theorem}[]

\newtheorem{proposition}[]{Proposition}

\newtheorem{lemma}[]{Lemma}

\theoremstyle{definition}

\newtheorem{remark}{Remark}

\newcommand{\var}{\mathrm{Var}}
\newcommand{\eps}{\varepsilon}

\usepackage{todonotes}

\newcommand{\R}{\mathbb{R}}

\title{\textsc{The critical Karp--Sipser core of random graphs}}
\author{
Thomas \textsc{Budzinski}\thanks{ ENS de Lyon and CNRS.\hfill  \href{mailto:thomas.budzinski@ens-lyon.fr}{\texttt{thomas.budzinski@ens-lyon.fr}}}
\qquad\&\qquad
Alice \textsc{Contat}\thanks{Universit\'e Paris-Saclay.\hfill  \href{mailto:alice.contat@universite-paris-saclay.fr}{\texttt{alice.contat@universite-paris-saclay.fr}}}
\qquad\&\qquad
Nicolas \textsc{Curien}\thanks{Universit\'e Paris-Saclay.\hfill  \href{mailto:nicolas.curien@gmail.com}{\texttt{nicolas.curien@gmail.com}}}
}

\date{}
\begin{document}
	\maketitle 
	
	\abstract{We study the Karp--Sipser core of a random graph made of a configuration  model with vertices of degree $1,2$ and $3$. This core is obtained by recursively removing the leaves as well as their unique neighbors in the graph. We settle a conjecture of Bauer \& Golinelli \cite{bauer2001core} and prove that at criticality, the Karp--Sipser core has size $ \approx \mathrm{Cst} \cdot   \vartheta^{-2}   \cdot n^{3/5}$ where $\vartheta$ is the hitting time of the curve $t \mapsto \frac{1}{t^{2}}$ by a linear Brownian motion started at $0$. Our proof relies on a detailed multi-scale analysis of the Markov chain associated to the Karp-Sipser leaf-removal  algorithm close to its extinction time. 	}
	
	\begin{figure}[!h]
		\begin{center}
			\includegraphics[height=5cm]{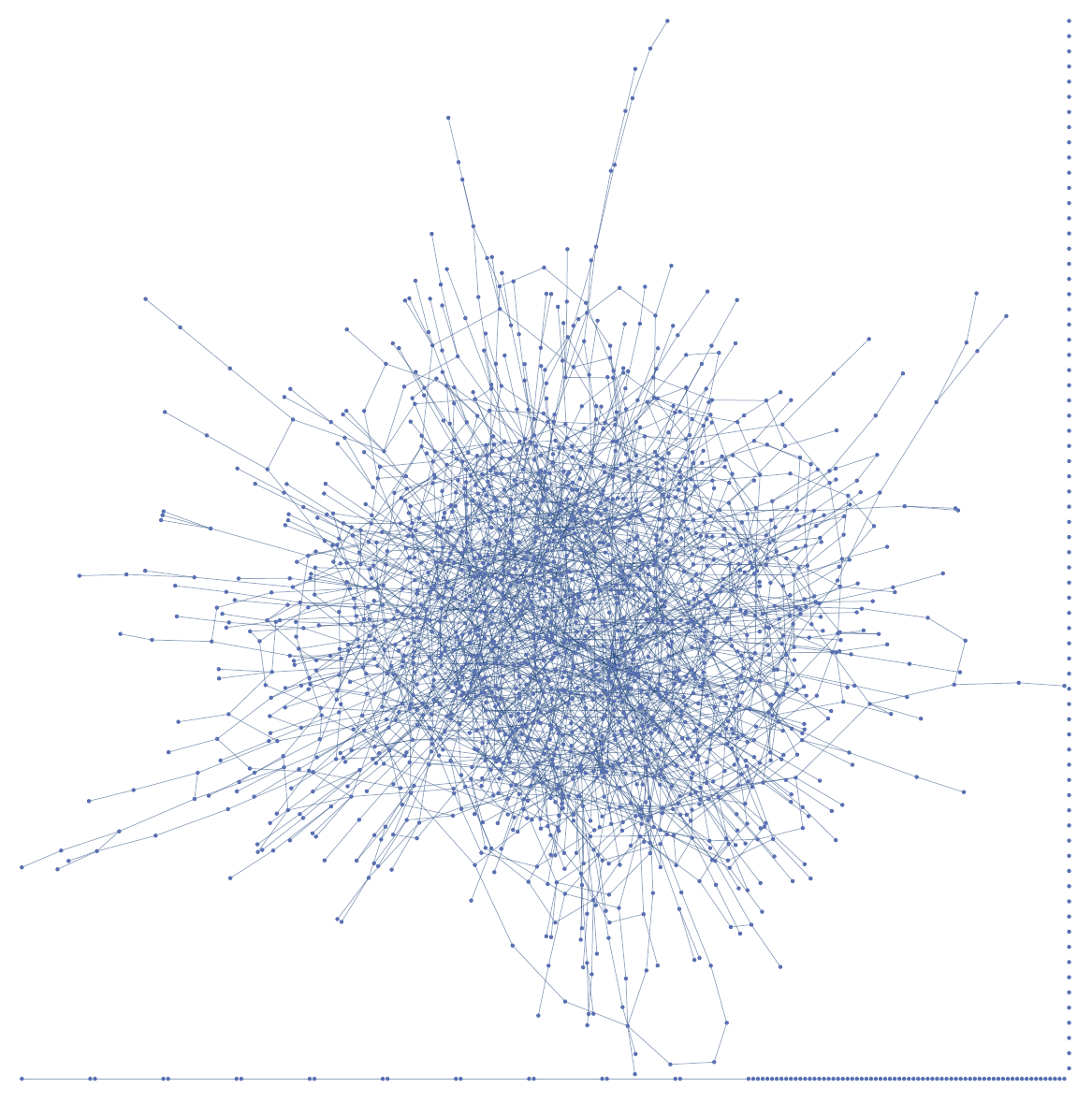} \hspace{0.5cm}
			\includegraphics[height=4cm]{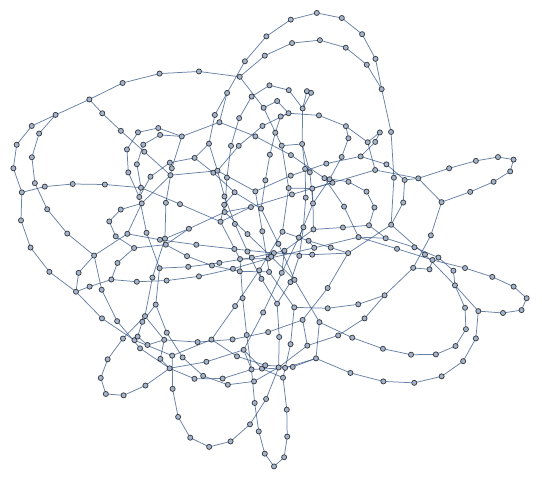}\hspace{0.5cm}
			\includegraphics[height=5cm]{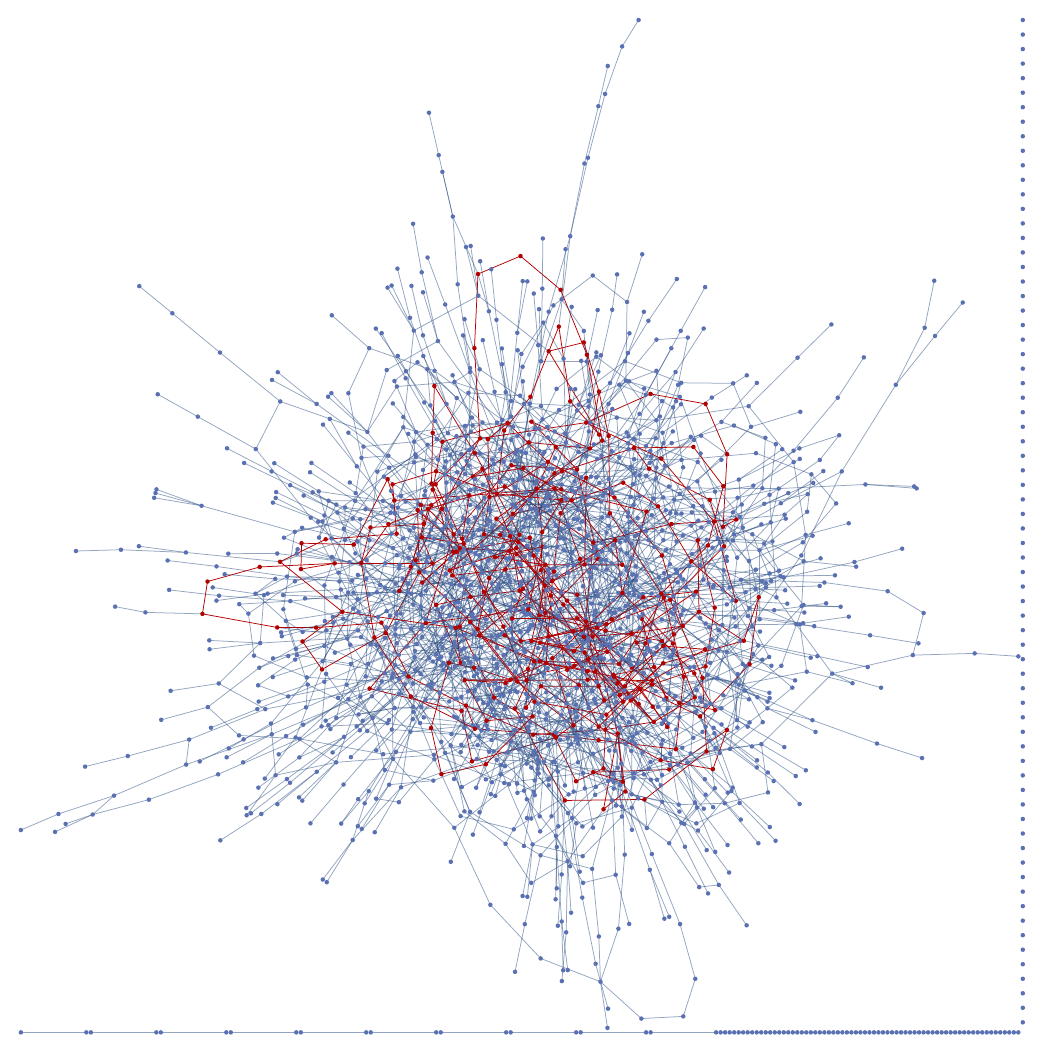}
			\caption{ \textbf{(Left).} The giant component of an Erd{\H{o}}s--R\'enyi random graph $ \mathrm{G}(n, \frac{ \mathrm{e}}{n})$ with $n=2000$ on the left and \textbf{(Middle)} its Karp--Sipser core. \textbf{(Right).} The Karp--Sipser core in red inside the original graph. \label{fig:ERcore}}
		\end{center}
	\end{figure}
	\section {Introduction}
	
	\paragraph{The Karp--Sipser algorithm.}
	Let $ \mathfrak{g}$ be a finite graph. The Karp--Sipser algorithm \cite{karp1981maximum} consists in removing recursively the vertices of degree $1$ in $ \mathfrak{g}$ as well as their unique neighbors \rev{and removing the isolated vertices that may appear in the process}, see Figure \ref{fig:KSillustration}.  The initial motivation of Karp \& Sipser for considering this algorithm is that the leaves\footnote{Here and in the rest of the paper, the concept of leaf is a dynamical concept, as a vertex in the initial graph which is not a leaf may become one later.} and isolated vertices removed during this process form an independent set of $ \mathfrak{g}$ which has very high density. We recall that an independent set in $ \mathfrak{g}$ is a subset of vertices, no two of which are adjacent. The problem of finding an independent set of maximal size is in general a NP-hard problem \rev{\cite{frieze1997algorithmic}}, and the Karp--Sipser algorithm provides a fair lower bound. \rev{More precisely, it is ``optimal in the beginning'' in the sense that there is an independent set with maximal size that contains all the leaves removed by the algorithm (before the first time where there are no leaves left)}. \rev{More generally, greedy strategies are a natural way to approximately solve optimization problems on (random) graphs in a way that is computationally efficient (and sometimes quasi-optimal up to a certain threshold) illustrating the famous ``Greed is good" concept  \cite{tropp2004greed,halldorsson1994greed}. See e.g.~\cite{BJLS18} for  a recent application of a degree-greedy strategy to Wifi protocols.}
	\begin{figure}[!h]
		\begin{center}
			\includegraphics[width=16cm]{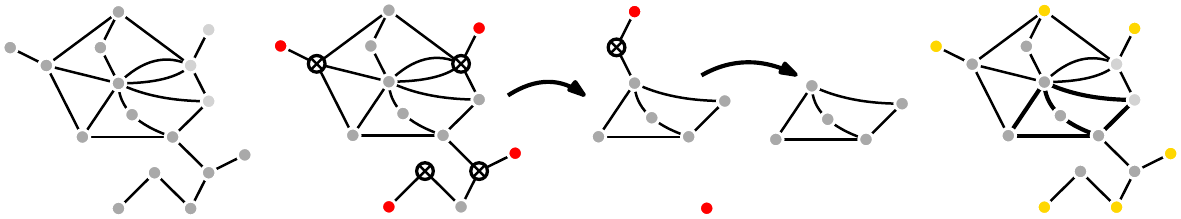}
			\caption{Illustration of the Karp--Sipser algorithm. The first 4 figures show the initial graph, as well as the recursive deletion process of the leaves (in red) together with their unique neighbor (crosses), until no leaf is left: we then obtain the Karp--Sipser core (fourth figure). On the right, the initial graph is represented together with the Karp--Sipser core in thick lines and the independent set formed by the removed ``leaves'' in yellow. \label{fig:KSillustration}}
			
		\end{center}
	\end{figure}

	\paragraph{The Karp--Sipser core of random graphs.}
	A striking property of the leaf-removal process is its Abelian property: whatever the order in which we decide to recursively remove the leaves and their neighbors, we always obtain the same subgraph of $ \mathfrak{g}$ (with no leaves) which we will call the \emph{Karp--Sipser core} of $ \mathfrak{g}$ and denote by $ \mathrm{KSCore}( \mathfrak{g})$, see \cite[Appendix]{bauer2001core} or \cite[Section 1.6.1]{kreacic2017some}. Beware that the above notion differs from the usual $k$-core of a graph\footnote{The $k$-core of $ \mathfrak{g}$ is the largest subset $V$ of its vertices such that for any $v \in V$, the induced degree of $v$ within $V$ is at least $k$.}, see Section \ref{sec:kcore}. By the above remark, the Karp--Sipser algorithm creates an independent set (the leaves removed during the algorithm) whose size is within at most $| \mathrm{KSCore}( \mathfrak{g})|$ from the maximal size of an independent set in $ \mathfrak{g}$.
	
	The performance of the Karp--Sipser algorithm on the Erd{\H{o}}s--R\'enyi random graph $ \mathrm{G}(n, \frac{c}{n})$ has been analyzed  in the pioneer\rev{ing} work \cite{karp1981maximum} and later refined in the breakthrough work~\cite{aronson1998maximum} which established a phase transition as $n \to \infty$ depending on the value of $c$: 
	\begin{itemize}
		\item if $c < \mathrm{e}$, then as $n \to \infty$, the size $| \mathrm{KSCore}( \mathrm{G}(n, \frac{c}{n}))|$ is of order $O(1)$;
		\item if $c > \mathrm{e}$, then as $n \to \infty$, the size $| \mathrm{KSCore}( \mathrm{G}(n, \frac{c}{n}))|$ is of order $n$.
	\end{itemize}
	Those works have later been extended to the configuration model \cite{bohman2011karp,jonckheere2021asymptotic}. However, the careful analysis of the critical case $ c =  \mathrm{e}$ was open as of today to the best of our knowledge. In \cite{bauer2001core}, based on numerical simulations, the physicists Bauer \& Golinelli predicted that $| \mathrm{KSCore}( \mathrm{G}(n, \frac{ \mathrm{e}}{n}))|$ should be of order $n^{3/5}$. The main result of this work (Theorem \ref{thm:maincritical}) is to settle this conjecture in the case of a random graph with degrees $1,2$ and $3$.
	
	\paragraph{Model and results.} In this paper we shall consider a random graph model closely related to $ \mathrm{G}(n, \frac{c}{n})$ but for which the analysis of the Karp--Sipser algorithm is simpler. Namely, we fix a sequence of numbers $ \mathbf{d}^n= (d_1^n, d_2^n,d_3^{n})_{n \geq 1}$ such that $$n = d_{1}^{n} + 2 d_{2}^{n} + 3 d_{3}^{n} \mbox{ is even}.$$ We imagine $ \mathbf{d}^n$ as the number of vertices of degree $1,2$ and $3$ and consider a random multi-graph $ \mathrm{CM}(  \mathbf{d}^{n})$ sampled by pairing the edges emanating from the $d_{1}^{n} + d_{2}^{n} + d_{3}^{n}$ vertices uniformly at random. This is a special instance of the so-called configuration model introduced by Bollobas \cite{Bollobas}, see \cite{RemcoRGII} for background. In the rest of the paper we shall further  assume that 
	\begin{eqnarray} \label{eq:convprop} \frac{d_{1}^{n}}{n} \xrightarrow[n\to\infty]{} p_{1},\quad  \frac{2d_{2}^{n}}{n} \xrightarrow[n\to\infty]{} p_{2}, \quad \mbox{ and } \quad \frac{3d_{3}^{n}}{n} \xrightarrow[n\to\infty]{} p_{3},  \end{eqnarray}
	\rev{where $p_1, p_2, p_3 \geq 0$,} so that the proportion of half-edges which are incident to a vertex of degree $i$ is $p_i$. Our goal will be to analyze $ \mathrm{KSCore}( \mathrm{CM}(  \mathbf{d}^{n}))$. A phase transition has been observed in \cite{jonckheere2021asymptotic} for the size of the Karp--Sipser core but its location depending on $  (p_{1}, p_{2}, p_{3})$ was not explicit. Our first contribution is to make this threshold precise. For a multi-graph $\mathfrak{g}$, we will write $|\mathfrak{g}|$ for twice the number of edges of $\mathfrak{g}$, and call this quantity the \emph{size} of $\mathfrak{g}$. If $(u_n)$ is a sequence of positive numbers and $(X_n)$ a sequence of random variables, we will write $X_n=O_{\mathbb{P}}(u_n)$ if $\left( u_n^{-1} X_n \right)$ is tight, and we will write $X_n=o_{\mathbb{P}}(u_n)$ if $u_n^{-1} X_n$ converges to $0$ in probability.
	
	\begin{theorem}[Explicit phase transition] \label{thm:phasetransition} Under the assumptions \eqref{eq:convprop}, let 
	\begin{equation}\label{eq:defTheta}\Theta = (p_3-p_1)^2 - 4p_1.\end{equation}
		\begin{itemize}
			\item \textbf{Subcritical phase.} If $\Theta<0$, 
			 then  as $n \to \infty$ we have $$ |\mathrm{KSCore}( \mathrm{CM}(  \mathbf{d}^{n})) | = O_{ \mathbb{P}}(\log^2 n).$$ 
			\item  \textbf{Supercritical phase.} If $\Theta>0$, 
			 then
			$$ n^{-1} \cdot |\mathrm{KSCore}( \mathrm{CM}(  \mathbf{d}^{n}))|  \xrightarrow[n\to\infty]{( \mathbb{P})}   \frac{4 \Theta}{3+ \Theta}.$$
			\item \textbf{Critical phase.} If $\Theta =0$,
			 then  $ |\mathrm{KSCore}( \mathrm{CM}(  \mathbf{d}^{n}))| = o_{ \mathbb{P}}(n)$.
		\end{itemize}
	\end{theorem}
	\rev{We note that by the criterion of~\cite{molloy1995critical}, there is a giant component in $\mathrm{CM}(  \mathbf{d}^{n})$ if and only if $p_3>p_1$, so the phase transition is distinct from the ``classical" giant component transition.}	
	
	\paragraph{Sketch of proof of the phase transition.}
	The proof of this theorem uses classical techniques. We shall reveal the random graph $ \mathrm{CM}( \mathbf{d}^{n})$ by pairing its half-edges two-by-two as we perform the Karp--Sipser leaf removal algorithm (a.k.a.~peeling algorithm). More precisely, when we remove a leaf, we reveal its neighbor in the graph and remove it as well, which decreases the degrees of some other vertices. During this process, the number of remaining vertices of degree $1,2$ and $3$ evolves as a $(\mathbb{Z}_{\geq 0})^3$-valued Markov chain with explicit \rev{transition probabilities}. This is, of course, a recurrent idea in random graph theory and has already been used many times for the Karp--Sipser algorithm itself \cite{karp1981maximum,aronson1998maximum}. More precisely, we shall erase leaves uniformly at random one-by-one (in contrast with  \cite{jonckheere2021asymptotic}, where all possible leaves are erased at each round) and use the \emph{differential equation method} \cite{wormald1995differential} to prove that the renormalized number of vertices of degree $1,2$ and $3$ is well approximated by a differential equation on $ \mathbb{R}^{3}$ for which we are able to find explicit solutions. In a sense, this returns to the roots of this method since it was Karp \& Sipser \cite{karp1981maximum} who first introduced it in the context of random graphs following earlier works of Kurtz \cite{kurtz1970solutions} in population models. 
	
	\begin{remark}[A spectral parallel to the Karp--Sipser phase transition] \label{rek:CS}
		The \emph{nullity} of a graph is the multiplicity of $0$ in the spectrum of its adjacency matrix. It is easy to see that the leaf-removal process on a graph $ \mathfrak{g}$ leaves its \emph{nullity} invariant and so the Karp--Sipser algorithm can also be used to study the \rev{latter}, see \cite{bauer2001random,salez2011some}. The phase transition for the emergence of a Karp--Sipser core of positive proportion in $ \mathrm{G}(n, \frac{ \mathrm{e}}{n})$ has a parallel phase transition for the emergence of extended states (an absolutely continuous part) at zero in $ \mathrm{G}(n, \frac{ \mathrm{e}}{n})$, see \cite{bauer2001random,coste2021emergence}. \rev{More precisely, if $G_n$ is a random graph of size $n$, we can associate with it the empirical eigenvalues distribution 
		$$ \mu_{G_n} = \sum_{i=1}^n \delta_{\lambda_i},$$ where $\lambda_1 \geq ... \geq \lambda_n$ are the real eigenvalues of its (symmetric) adjacency matrix. When $G_n =\mathrm{G}(n, \frac{c}{n})$, an adaptation of the celebrated result of Wigner states that $\mu_{G_n}$ converges towards a deterministic but non-explicit measure $\mu_c$, which is the expected spectral measure of the (unimodular) Poisson--Galton--Watson tree with mean $c$, see \cite{bordenave2010resolvent}. Although $\mu_c$ is quite mysterious, Bauer and Golinelli conjectured that it undergoes a phase transition at $c = \mathrm{e}= 2,718...$: if $c < \mathrm{e}$ then they predicted that
		$$ \lim_{ \varepsilon \to 0} \frac{\mu_c([- \varepsilon, \varepsilon]) - \mu_c(\{0\})}{ \varepsilon} =0,$$ in which case $\mu_c$ is said to have non-extended states at $0$, and if $c> \mathrm{e}$ then $\mu_c$ has extended states at $0$. This conjecture was proved by Coste and Salez \cite{coste2021emergence} who more generally gave an explicit criterion to decide whether the expected spectral measure $\mu_\pi$ of a unimodular Galton--Watson tree with offspring distribution $\pi$ possesses or not an extended state at $0$.  In particular, if $G_n = \mathrm{CM}( \mathbf{d}^n)$ under assumption \eqref{eq:convprop}, then it is easy to see that $G_n$ converges in the Benjamini--Schramm sense towards the unimodular version of a Galton--Watson tree with offspring distribution $ \pi = p_1 \delta_0 + p_2 \delta_1  + p_3 \delta_2$, and consequently $\mu_{G_n}$ converges towards $\mu_\pi$ by \cite{bordenave2010resolvent}. In this case, an easy calculation shows that the Coste--Salez criterion for existence of an extended state for $\mu_\pi$ is precisely $\Theta >0$ where $\Theta$ is defined by \eqref{eq:defTheta} in Theorem \ref{thm:phasetransition}. It is very tempting to conjecture that the ``continuous" part of the spectrum near $0$ in $\mu_\pi$ is indeed created by the giant component of the Karp--Sipser core, but unfortunately, we have no geometric nor spectral explanation of the coincidence of the above two thresholds which come from two very different computations.}
	\end{remark}
	
	We now turn to the detailed analysis of the \emph{critical case} which is the main goal of our work. For this we fix a particular degree sequence $ \mathbf{d}_{ \mathrm{crit}}^{n} = (d_{1,c}^{n},d_{2,c}^{n},d_{3,c}^{n})$ such that $d_{1,c}^{n} + 3 d_{3,c}^{n}=n$ is even (to be able to perform the configuration model) and
	\begin{eqnarray} \label{eq:convcrit} d_{1,c}^{n} = n \left( 1-\frac{ \sqrt{3}}{2} \right) + O(1), \quad 2 d_{2,c}^{n} =0 , \quad \mbox{ and } 3 d_{3,c}^{n} = n \frac{ \sqrt{3}}{2} +O(1).  \end{eqnarray} In particular we have $\Theta = ( \sqrt{3}-1)^{2}-4(1- \frac{ \sqrt{3}}{2})=0$ so we are indeed in the critical case of Theorem \ref{thm:phasetransition}. By definition, the core $\mathrm{KSCore}( \mathrm{CM}( \mathbf{d}_{ \mathrm{crit}}^{n}))$ has only vertices of degrees $2$ or $3$. Our main result is then the following:
	\begin{theorem}[Geometry of the critical Karp--Sipser core] \label{thm:maincritical} Let $D_{2}(n)$ (resp. $D_{3}(n)$) be the total number of half-edges attached to a vertex of degree $2$ (resp. $3$) in $\mathrm{KSCore}( \mathrm{CM}( \mathbf{d}_{ \mathrm{crit}}^{n}))$. Then we have 
		$$ \left(\begin{array}{c}n^{-3/5} \cdot D_{2}(n) \\  n^{-2/5} \cdot D_{3}(n) \end{array}\right)  \xrightarrow[n\to\infty]{(d)} \cdot \left( \begin{array}{c}3^{-3/5}2^{14/5} \cdot \vartheta^{-2} \\  
		3^{-2/5}2^{16/5} \cdot \vartheta^{-3} \end{array}\right),$$ 
		where $ \vartheta =  \inf \{ t \geq 0 : \, B_t=t^{-2}\},$
		 for a standard linear Brownian motion  $(B_t: t \geq 0)$ started from $0$. 	Moreover, conditionally on $(D_{2}(n),D_{3}(n))$, the graph $ \mathrm{KSCore}( \mathrm{CM}( \mathbf{d}_{ \mathrm{crit}}^{n}))$ is a configuration model.\end{theorem}

\begin{remark}[Bauer \& Golinelli's prediction] The above theorem confirms a long-standing prediction of Bauer \& Golinelli \cite{bauer2001core} stated in the case of the Erd{\H o}s-R\'enyi random graph: based on Monte-Carlo simulations they proposed a few possible sets of critical exponents \cite[Table 1]{bauer2001core} and our theorem confirms their prediction. See also \cite{goldschmidt2019spread,kreacic2017some} for later developments. 
\end{remark}

	Note that our assumptions on the initial degree sequence are much stronger than for Theorem~\ref{thm:phasetransition} since the size of the critical core is quite sensitive to initial conditions. Our proof still works if \rev{the triple $\left( 1-\frac{\sqrt{3}}{2}, 0, \frac{\sqrt{3}}{2}\right)$ in~\eqref{eq:convcrit} is replaced by a triple $(p_1, p_2, p_3)$ that is critical in the sense of Theorem~\ref{thm:phasetransition} (this is equivalent to starting "later" along the critical curve of Figure~\ref{fig:equadiff_trajectories} below)}. Moreover, the error $O(1)$ can be replaced by $O(n^{1/2})$, and the result should remain true as long as the initial error is $o(n^{3/5})$, see Section~\ref{subsec:near_critical} for a discussion on the near-critical regime.
	\rev{On the other hand, the reason why we restricted ourselves to vertex degrees bounded by $3$ is that it is the regime where we could find explicit solutions to the differential equation which appears in the scaling limit. However,} we believe that the above limiting result holds \rev{(possibly with different constant factors)} for a large variety of random graphs which are critical for the Karp--Sipser algorithm. In particular, we expect a similar result for configuration models with bounded degrees and for the Erd{\H{o}}s--R\'enyi graph \rev{(the latter is a work in progress of the first two authors)}.

	\paragraph{Ideas of proof.} The proof of Theorem~\ref{thm:maincritical} uses the same Markov chain as the one used to study the phase transition. \rev{That is, at each step of the Karp--Sipser algorithm, the Markov chain is a triple of integers giving the number of vertices of degree $1$, $2$ and $3$ in the unexplored part of the graph.} The difference is that we need to study \rev{at a much finer scale} the behaviour of this chain right before \rev{its extinction time, i.e. the first time where the number of leaves hits $0$.} More precisely, we can expect from the differential equation approximation that $\eps n$ steps before extinction, the number of vertices of unmatched degrees $1$, $2$ and $3$ are respectively of order $\eps^2 n$, $\eps n$ and $\eps^{3/2} n$. On the other hand, a variance computation shows that the fluctuations of the number of vertices of degree $1$ are of order $\eps^{3/4} \sqrt{n}$. Finally, the time at which we can expect the Markov chain to terminate is the time when the fluctuations exceed the expected value, that is at $\eps=n^{-2/5}$. However, checking that the differential equation approximation remains good until that scale requires some careful control of the Markov chain across scales. In particular, the reason why the fluctuations become much smaller than $\sqrt{n}$ in the end of the process is that the drift of our Markov chain induces a ``self-correcting" effect.
	
	\rev{
	\paragraph{Level of generality of the method.} While Theorem~\ref{thm:maincritical} is limited to a quite specific model, we believe that the techniques developed in its proof could more generally be used to understand precisely the exit times of Markov chains from domains.
	To fix ideas, let $ (\mathbb{X}^n : k \geq 0)$ be a $ \mathbb{Z}^d$-valued Markov chain whose expected conditional drift is well-approximated by $\phi( \mathbb{X}^n/n)$ for some function $\phi : \mathbb{R}^d \to \mathbb{R}^d$. The differential equation method shows that under some mild assumptions $ (n^{-1}  \mathbb{X}^n_{\lfloor tn \rfloor} : t \geq 0)$ converges towards a solution $ \mathcal{X}$ to $ \mathcal{X}'(t) = \phi( \mathcal{X}(t))$. If $\Omega$ is a bounded domain and $\Omega^n$ is its discrete approximation, it is reasonable to believe that the exit time $\theta^n$ of $\Omega^n$ by $ \mathbb{X}^n$ should converge after normalization towards the exit time $ t_{ \mathrm{ext}}$ of $\Omega$ by $ \mathcal{X}$. However, the fine fluctuations of $\theta^n$ around $ n t_{ \mathrm{ext}}$ should depend on fine properties of $\phi$ (and its derivatives) near the exit point. In the ``generic" case, it is natural to expect $\theta^n$ to satisfy a central limit theorem, with fluctuations of order $\sqrt{n}$. On the other hand, our techniques allow to study precisely the fluctuations of $\theta^n$ in a ``non-generic" setting, where the gradient of $\phi$ is tangent to $\partial \Omega$ at the exit point or in the presence of saddle points as in \cite{turner2007convergence}. Developing a general result should have applications to many other problems. Two natural examples are the study of the $k$-core of random graphs and the critical Karp--Sipser core of the more natural Erd\"os--R\'enyi random graph model, which will be the object of a future work by the first two authors. We refer to Section~\ref{sec:comments} for a discussion on these problems.
	}
	
	\bigskip
		
	\noindent \textbf{Acknowledgments.} The last two authors were supported by ERC 740943 GeoBrown and by ANR RanTanPlan. The first author is grateful to the Laboratoire de Mathématiques d'Orsay, where most of this work was done, for its hospitality. We warmly thank Matthieu Jonckheere for a stimulating discussion about  \cite{jonckheere2021asymptotic}  and Justin Salez for enlightening explanations about maximal matchings and independent sets in random graphs. \rev{We are grateful to Adrianus Twigt who kindly checked that the Coste--Salez threshold coincides with the Karp--Sipser threshold in our configuration model, see Remark \ref{rek:CS}. We also thank two anonymous referees for their helpful remarks.}

	\section{Karp--Sipser exploration of the configuration model}
	As we mentioned in the introduction the main idea (already present in \cite{karp1981maximum,aronson1998maximum,jonckheere2021asymptotic,bohman2011karp,kreacic2017some}) is to explore the random configuration model $ \mathrm{CM}( \mathbf{d}^{n})$ at the same time as we run the Karp--Sipser algorithm to discover its core. Let us explain this in details. Fix a degree sequence $ \mathbf{d}^{n}=( d_{1}^{n}, d_{2}^{n}, d_{3}^{n})$ such that $n= d_1^n+2d^n_2+3d_3^n$ is even. We shall expose the $ \frac{n}{2}$ edges of $ \mathrm{CM}( \mathbf{d}^{n})$ one by one and create a process 
	$$(X^n_k,Y^n_k,Z^n_k : k \geq 0)$$
	where $X^{n}, Y^{n}, Z^{n}$ represent respectively the number of unmatched half-edges linked to vertices of \emph{unmatched degree} $1,2,3$ \rev{(the unmatched degree of a vertex at time $k$ is the number of half-edges attached to this vertex which are still unmatched at time $k$).} The process of the sum is denoted by $S^{n} = X^{n} + Y^{n} + Z^{n}$. In particular, we always have $(X^{n}_0, Y^{n}_0, Z^{n}_0) = (d_{1}^{n}, 2 d_{2}^{n}, 3 d_{3}^{n})$ and $S^n_0=n$ with our conventions. 
	
	As long as $X^{n}_k >0$, the process evolves as follows. Since $X^{n}_k >0$, there are still vertices of unmatched degree $1$. We pick $\ell$ (for leaf) one of these vertices  uniformly at random and reveal its neighbor $v$ in the graph. Now, in the Karp--Sipser algorithm this vertex is ``destroyed'' so we shall erase $v$ from the configuration \emph{as well as the connections \rev{it} has with other vertices of the graph}. More precisely, we reveal the neighbors of $v$ in $\mathrm{CM}(  \mathbf{d}^{n})$ and erase all the connections we create when doing so. In particular, if $v$ is connected to a vertex $w \ne \ell$ of unmatched degree $d$ via $i$ edges, then after the operation $w$ becomes a vertex of unmatched degree $d-i$. After that, the vertices of unmatched degree $0$ are simply removed. We listed all $13$ combinatorial  possibilities (recall that our vertices have degree $1,2$ or $3$) in Figure \ref{fig:table}. The stopping time of the algorithm is  $$\theta^{n} := \inf \{ k \geq 0 : X^{n}_k =0\}.$$ Finally, we extend the process $(X^n,Y^n,Z^n)$ to any $k$ by setting $(X^n_k,Y^n_k,Z^n_k)=(X^n_{\theta^n},Y^n_{\theta^n},Z^n_{\theta^n})$ for $k \geq \theta^n$. We denote by $(\mathcal{F}_k)_{k \geq 0}$ the natural filtration  generated by this exploration. The starting point of our investigations is the following.
	\begin{proposition}\label{prop:markovexplo} The process $(X^n_k, Y^n_k,Z^n_k)_{0 \leq k \leq \theta^n}$ is a Markov process whose \rev{transition probabilities} are described in Figure \ref{fig:table}. Furthermore, for any stopping time $\tau$, the remaining pairing of the unmatched edges conditionally on $\mathcal{F}_{\tau}$ is uniform.
	\end{proposition}
	\begin{proof} \rev{In words, the proposition says that we can construct the random graph $ \mathrm{CM}( \mathbf{d}^n)$ \textit{at the same time as} we perform the leaf-removal algorithm to reveal its Karp--Sipser core. This is a classical idea in random graph theory, but since this is the crux of the approach, we give a few details. Imagine that the degree sequence $ \mathbf{d}^n$ is fixed and label arbitrarily  the $n=1 d_1^n + 2 d_2^n + 3 d_3^n$ legs (half-edges) incident to the $d_1^n + d_2^n + d_3^n$ vertices. Consider then a uniform matching $ \mathcal{M}$, i.e. an involution of $\{1,2, ... , n\}$ without fixed points chosen  uniformly among the $ (n-1)!!$ possibilities. The graph $ \mathrm{CM}( \mathbf{d}^n)$ is then obtained by pairing the legs according to the matching $ \mathcal{M}$. We will now use repeatedly the following elementary fact: suppose that $I \in \{1,2, ... , n\}$ is a random index independent of $ \mathcal{M}$ and denote by $J \in \{1,2, ... , n\}$ its image/pair in $ \mathcal{M}$. Then $J$ is uniformly distributed over $\{1, ... ,n\} \backslash \{I\}$ and conditionally on $(I,J)$, the restriction of $ \mathcal{M}$ to $\{1, ... ,n\} \backslash \{I,J\}$ is (after relabeling) a uniform matching. In terms of the random graph $\mathrm{CM}( \mathbf{d}^n)$ this means that if we destroy the edge associated to a leg selected independently from the underlying matching, then conditionally on the remaining degree sequence $\mathbf{d}^{n-1}$, the resulting graph has law $ \mathrm{CM}(\mathbf{d}^{n-1})$. We can then iterate several times this property to perform one step of the leaf-removal in the Karp--Sipser algorithm: first delete the edge associated to a leg attached to a vertex of degree $1$ (picked arbitrarily but independently of $\mathcal{M}$) and then delete the edges associated to the legs of the possible neighbor revealed. Conditionally on the resulting degree sequence, the rest of the graph is still a configuration model and the probability transitions are indeed given by Figure \ref{fig:transitions}. The remaining properties immediately follow.  }
	\end{proof}
	\begin{figure}[!h]
		\begin{center}
			\includegraphics[width=16cm]{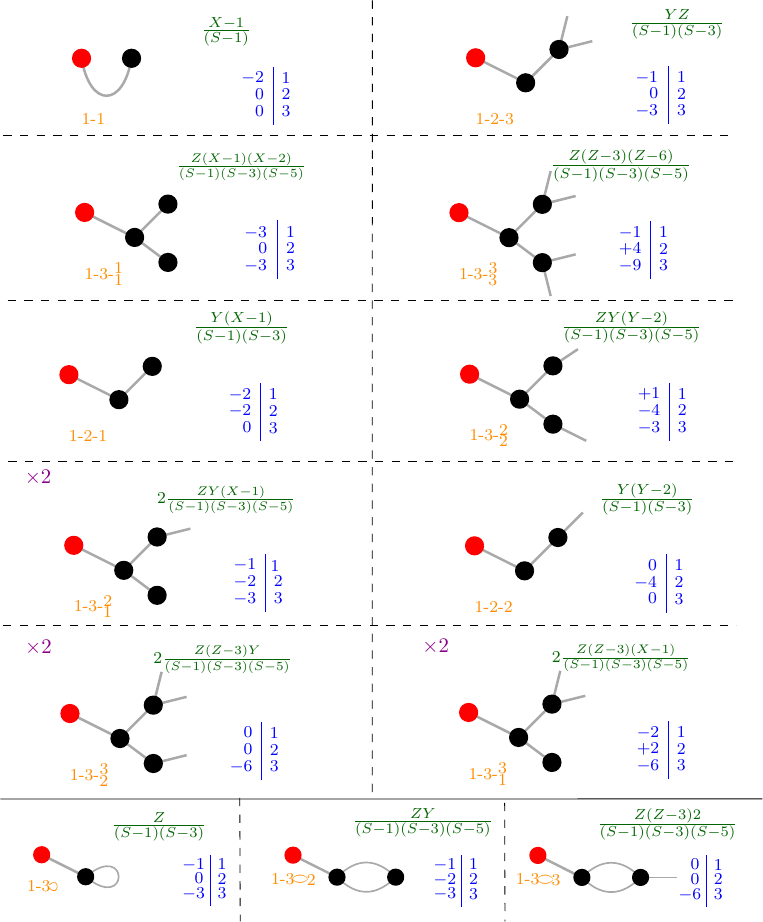}
			\caption{Transitions probabilities of the Markov chain $(X^{n}, Y^{n}, Z^{n})$: as long as $X^{n}>0$, a vertex $\ell$ of degree $1$ (in red above) is picked and its neighbor $v$ is revealed. The vertices $\ell,v$ are then removed from the configuration model as well as the connections they created. The probability of each event is indicated in green in the upper right corner \rev{and we recall that $S=X+Y+Z$}. The variation of $X,Y,Z$ are displayed in blue. A symmetry factor is indicated when needed in purple in the upper left corner. Notice in particular that the last three cases on the bottom have probabilities of smaller order $O(1/S)$, so they will not participate to the large scale limit. \label{fig:transitions}\label{fig:table}}
		\end{center}
	\end{figure}

	In particular, notice that at the stopping time $\theta^n$, the graph made by pairing the remaining unmatched edges is precisely the Karp--Sipser core of $ \mathrm{CM}(  \mathbf{d}^n)$ and so the second part of Theorem \ref{thm:maincritical} is already proved. 
		
	\section{Phase transition  \emph{via} fluid limit of the Markov chain}
	
	In this section, we prove Theorem \ref{thm:phasetransition}. The main ingredient is a deterministic fluid limit result for the  Markov chain $(X^{n},Y^{n},Z^{n})$. 
	
	\subsection{Fluid limit for the Markov chain}
For a process indexed by discrete time $(   \mathfrak{H}_k : k \geq 0)$ we use the notation $\Delta \mathfrak{H}_k = \mathfrak{H}_{k+1}- \mathfrak{H}_k$ for $k \geq 0$.	Given the \rev{transition probabilities} of the Markov chain $(X^{n}, Y^{n}, Z^{n})$ the following should come as no surprise.
	\begin{proposition}[Fluid limit] \label{prop:fluid-limit}Suppose that $ \mathbf{d}^{n}=(d_{1}^{n}, d_{2}^{n}, d_{3}^{n})$ satisfies \eqref{eq:convprop}. Then we have the following convergence in probability for the uniform norm:
		\begin{equation}\label{eq:fluid_limit}
		\left(\frac{X^n_{ \lfloor tn \rfloor}}{n}, \frac{Y^n_{ \lfloor tn \rfloor}}{n}, \frac{Z^n_{ \lfloor tn \rfloor}}{n}\right)_{0 \leq t \leq \theta^{n}/n} \xrightarrow[n\to\infty]{ (\mathbb{P})} ( \mathscr{X}(t), \mathscr{Y}(t), \mathscr{Z}(t))_{0 \leq t \leq t_{\ext}},
		\end{equation}
		where $( \mathscr{X}, \mathscr{Y}, \mathscr{Z})$ is the unique solution\footnote{More precisely, by \emph{solution}, we mean that $\left( \mathscr{X}, \mathscr{Y}, \mathscr{Z} \right)$ is a continuous function from $[0,t_{\ext}]$ to $\R^3$ such that $\mathscr{X}$ first hits $0$ at time $t_{\ext}$ and $\left( \mathscr{X}'(t), \mathscr{Y}'(t), \mathscr{Z}'(t) \right) = \phi \left( \mathscr{X}(t), \mathscr{Y}(t), \mathscr{Z}(t) \right)$ for all $0 \leq t < t_{\ext}$.} to the differential equation $(
		\mathscr{X}', \mathscr{Y}', \mathscr{Z}')
		= \phi( 
		\mathscr{X}, \mathscr{Y}, \mathscr{Z})$ with $\phi$ defined below \eqref{eqn_diff_system_primo} with initial conditions $ (p_{1}, p_{2}, p_{3})$ and where $t_{\ext}$ is the first hitting time of $0$ by the continuous process $ \mathscr{X}$. Moreover, 
		$ \theta^{n}/n \to t_{\ext}$ in probability as $n \to \infty$.
	\end{proposition}
	
	\begin{proof} It is a standard application of the differential equation method. Indeed, the increments of the Markov chain $(X^n,Y^n,Z^n)$ are bounded and using the exact \rev{transition probabilities} (Figure \ref{fig:table}), the conditional expected drifts 
		\[\E \left[ \Delta X^n_k, \Delta Y^n_k, \Delta Z^n_k \, | \,  \mathcal{F}_k \right]\]
		converge for large values of $n$ towards $ \phi \left( \frac{X^n_k}{n}, \frac{Y^n_k}{n}, \frac{Z^n_k}{n} \right)$ where the function $\phi$ is defined by
		\begin{equation}\label{eqn_diff_system_primo}
		\phi\left(\begin{array}{c}
		\mathscr{X}\\ \mathscr{Y}\\ \mathscr{Z}
		\end{array}\right) = \left( \begin{array}{c}
		-2 \mathbf{x} - \mathbf{y} \mathbf{z} - 3 \mathbf{x}^2 \mathbf{z} - 2 \mathbf{y} \mathbf{x} + \mathbf{z} \mathbf{y}^2 - 2 \mathbf{z} \mathbf{x} \mathbf{y} - 
		\mathbf{z}^3 - 4 \mathbf{z}^2 \mathbf{x}\\

		4 \mathbf{z}^3 - 2\mathbf{x} \mathbf{y} - 4 \mathbf{z} \mathbf{y}^2 - 4 \mathbf{x} \mathbf{y} \mathbf{z} - 4 \mathbf{y}^2 + 4 \mathbf{z}^2 \mathbf{x}\\
		
		- 3\mathbf{y} \mathbf{z} -3 \mathbf{z} \mathbf{y}^2 - 12 \mathbf{z}^2 \mathbf{y} -3 \mathbf{z} \mathbf{x}^2 - 6 \mathbf{x} \mathbf{y} \mathbf{z} - 
		12 \mathbf{z}^2 \mathbf{x} - 9 \mathbf{z}^3
		\end{array}\right) ,  \end{equation}
		\begin{equation}  \mbox{with } \mathscr{S}:= \mathscr{X}+ \mathscr{Y} + \mathscr{Z}  \mbox{ and where }   \left(\begin{array}{c}
		\mathbf{x}\\ \mathbf{y}\\ \mathbf{z }
		\end{array}\right) :=  \frac{1}{ \mathscr{S}}\left(\begin{array}{c}
		\mathscr{X}\\ \mathscr{Y}\\ \mathscr{Z}
		\end{array}\right) \mbox{ is the proportion vector}.
		\end{equation}
	For any $ \delta>0$, the convergence of the conditional expected  drifts to $\phi$ is uniform on  $\{ n^{-1}\cdot S^n \geq \delta\}$  and $ (x,y,z) \mapsto \phi(x,y,z)$ is  Lipschitz on $\{(x,y,z) \in \mathbb{R}_+^3 : \delta^{-1} \geq x+y+z \geq \delta\}$  as $\nabla \phi (x,y,z)$ is of the form $\frac{P(x,y,z)}{(x+y+z)^4}$, where $P$ is a polynomial. Therefore, by~\cite[Theorem 1]{wormald1995differential}, the equation $(
		\mathscr{X}', \mathscr{Y}', \mathscr{Z}')
		= \phi( \mathscr{X}, \mathscr{Y}, \mathscr{Z})$ with initial condition $(p_1,p_2,p_3)$ has a unique solution until the time $t_{\ext}^{\delta}$ where $\mathscr{X}$ first hits $\delta$, and the convergence~\eqref{eq:fluid_limit} holds for $0 \leq t \leq t_{\ext}^{\delta}$. Moreover, let $t_{\ext}=\lim_{\delta \to 0} t_{\ext}^{\delta}$. Since $\phi$ is bounded by an absolute constant, the solution $\left( \mathscr{X}, \mathscr{Y}, \mathscr{Z} \right)$ is Lipschitz on $[0,t_{\ext})$, so we can extend it uniquely in a continuous way to $[0,t_{\ext}]$, and by continuity $t_{\ext}$ is indeed the first time where $\mathscr{X}$ hits $0$.		We know that~\eqref{eq:fluid_limit} holds on every compact subset of $[0,t_{\ext})$.	Moreover, the increments of $(X^n,Y^n,Z^n)$ are bounded by an absolute constant, so the functions $n^{-1}\cdot (X^n,Y^n,Z^n)$ are uniformly Lipschitz and the previous convergence extends to a uniform convergence on $[0,t_{\ext}]$.
		
		We now only need to check that $\frac{\theta^n}{n}$ converges in probability to $t_{\ext}$. We notice that deterministically, if $k < \theta^n$, then $S^n_{k+1} \leq S^n_k-2$, which implies $\theta^n \leq n$, so up \rev{taking $n$ in some subsequence,} we may assume that $\frac{\theta^n}{n}$ converges to some random variable $\widetilde{t}_{\ext}$. By convergence of the process and the definition of $t_{\ext}$, it is immediate that $\widetilde{t}_{\ext} \geq t_{\ext}$. For the other direction, we treat two cases separately:
		\begin{itemize}
			\item if $\mathscr{S}(t_{\ext})=0$, then let $\eps>0$, and let $\delta>0$ be such that $\mathscr{S}(t_{\ext}-\delta)<\eps$. With probability $1-o(1)$ as $n \to +\infty$, we have $S^n_{\lfloor( t_{\ext}-\delta)n\rfloor}<2\eps n$. Since $S^n$ decreases by at least two at each step, this implies $\theta^n \leq ( t_{\ext}-\delta)n  + \eps n$, so $\widetilde{t}_{\ext} \leq t_{\ext}$.
			\item if $\mathscr{S}(t_{\ext})>0$, we first argue that the first component of $\phi \left( \mathscr{X}, \mathscr{Y}, \mathscr{Z} \right)$ remains bounded from above by a negative constant along the whole trajectory. Indeed, since $\mathscr{S}$ is bounded from below, we have $\mathscr{Z}' \geq - c \mathscr{Z}$ for some constant $c$ along the trajectory. Hence $\mathscr{Z}$ is bounded from below by a positive constant on $[0,t_{\ext}]$, so $\mathbf{y}$ is bounded away from $1$. Since the first component of $\phi \left( \mathscr{X}, \mathscr{Y}, \mathscr{Z} \right)$ is at most $-\mathbf{y}\mathbf{z}+\mathbf{y}^2\mathbf{z}=-\mathbf{y}\mathbf{z}(1-\mathbf{y})$, this proves our claim. Therefore, with high probability, the conditional expected drift $\E \left[ \Delta X^n_k \, | \, X^n_k \right]$ is also bounded from above by a negative constant $-c$ along the trajectory. Since the increments are bounded, by the weak law of large numbers this ensures $\widetilde{t}_{\ext} \leq t_{\ext}^{\eps}+\frac{1}{c}\eps$ for all $\eps>0$, so $\widetilde{t}_{\ext}=t_{\ext}$.
		\end{itemize}	
	\end{proof}
	
			\begin{figure}[!h]
		\begin{center}
			\includegraphics[width=5.5cm]{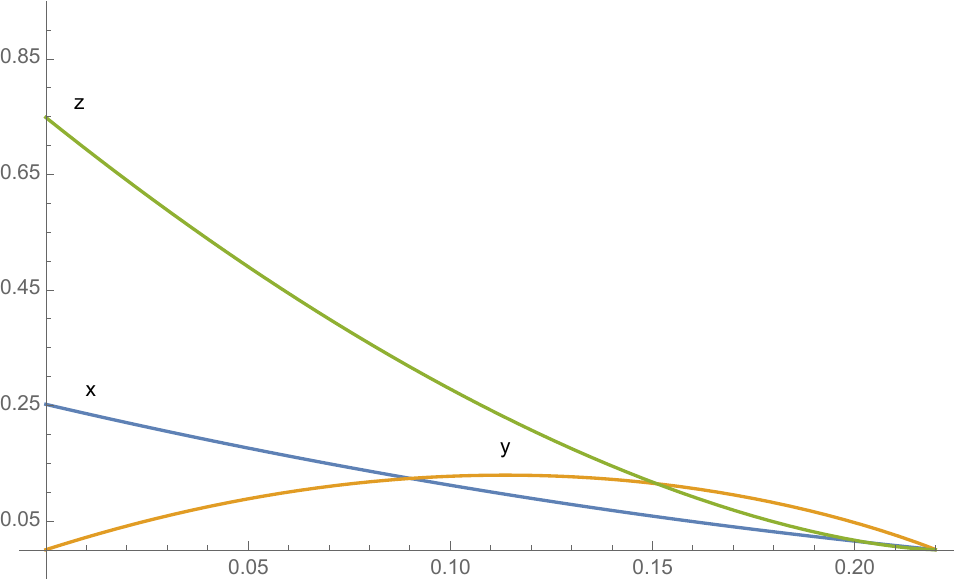}
			\includegraphics[width=5.5cm]{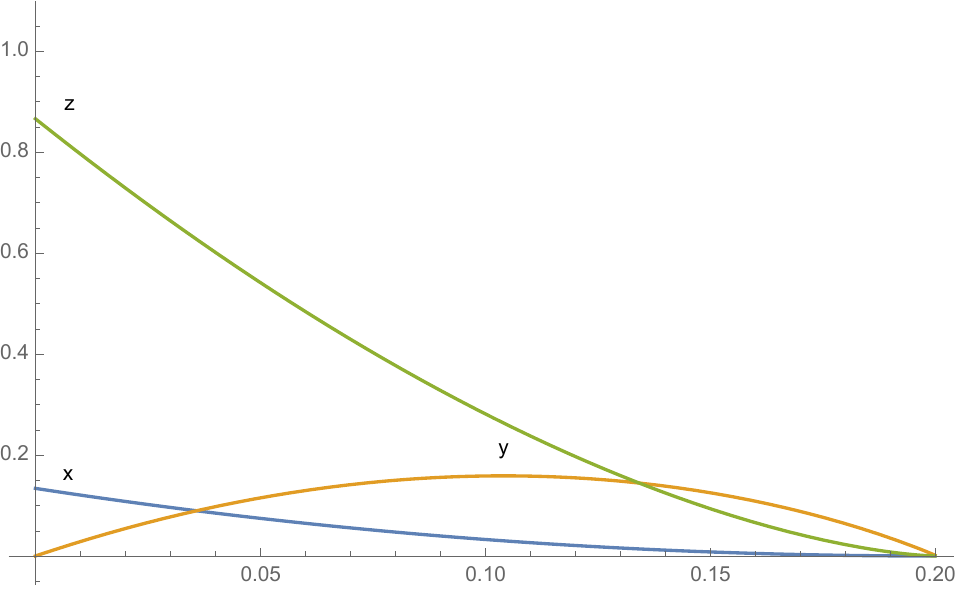}
			\includegraphics[width=5.5cm]{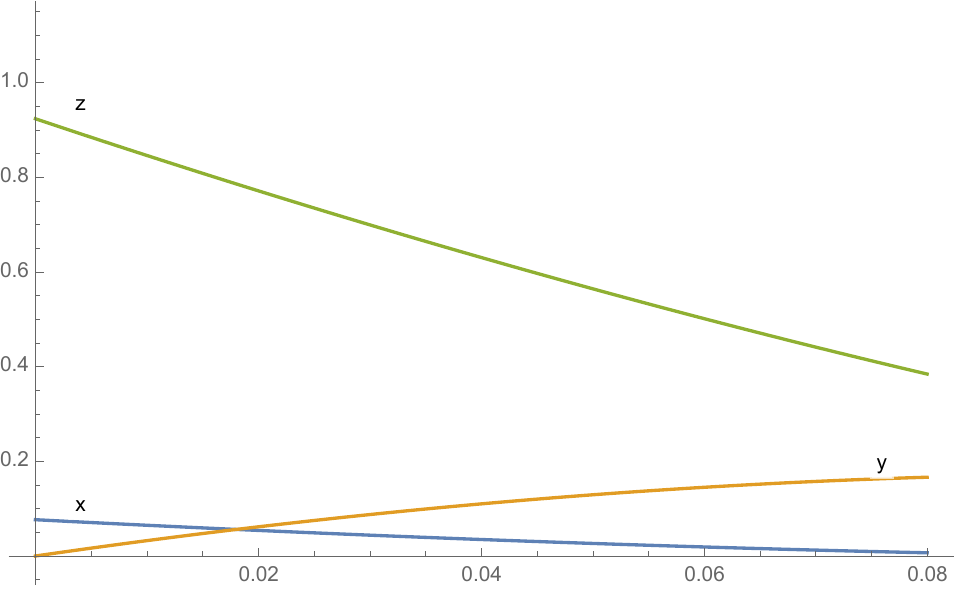}
			\caption{Illustration of the \rev{solution} $(\mathscr{X},\mathscr{Y},  \mathscr{Z})$ \rev{to the differential system} in terms of number of legs \rev{attached to vertices of degree $1$, $2$ and $3$} in the  subcritical (left), critical (center) and supercritical (right) cases.}\label{fig:equadiff_graphs}
		\end{center}
	\end{figure}
	\subsection{Solving the differential equation}
	
	In this section, our goal will be to gather information about the solutions to~\eqref{eqn_diff_system_primo}, which will give Theorem~\ref{thm:phasetransition} and be an important tool in the proof of Theorem~\ref{thm:maincritical}. As indicated by the system \eqref{eqn_diff_system_primo}, we will see that the solutions are easier to express in terms of proportions. We refer to Figures~\ref{fig:equadiff_graphs} and~\ref{fig:equadiff_trajectories} for a visualization of the trajectories of these solutions.
	
	\begin{proposition}\label{prop:solutions_equadiff}
		We fix $p_1>0$ and $p_2, p_3 \geq 0$ with $p_1+p_2+p_3=1$. Let $\left( \mathscr{X}(t), \mathscr{Y}(t), \mathscr{Z}(t) \right)_{0 \leq t \leq t_{\ext}}$ be the solution to~\eqref{eqn_diff_system_primo} with initial condition $(p_1,p_2,p_3)$.  Recall from \eqref{eq:defTheta} the definition $$\Theta = (p_3- p_1)^2 - 4 p_1 \in [-3,1].$$
		\begin{itemize}
			\item If $ \Theta <0$ (subcritical case
			), then $\mathscr{X}(t_{\ext})=\mathscr{Y}(t_{\ext})=\mathscr{Z}(t_{\ext})=0$. Moreover, for $t<t_{\ext}$ sufficiently close to $t_{\ext}$, we have $\mathscr{Z}(t)<\mathscr{X}(t)$.
			\item If $ \Theta >0$ (supercritical case
		), then 
			\begin{align}\label{eq:YZ_extinction_equadiff} \mathscr{X}(t_{\ext})=0, \quad 
			\mathscr{Y}(t_{\ext}) = \frac{\Theta }{p_3^2} \left( 1-\sqrt{\Theta} \right) >0, \quad \mbox{ and }  \quad 
			\mathscr{Z}(t_{\ext}) = \frac{\Theta}{p_3^2}\sqrt{\Theta} >0.
			\end{align}
			\item If $ \Theta =0$ (critical case), then $\mathscr{X}(t_{\ext})=\mathscr{Y}(t_{\ext})=\mathscr{Z}(t_{\ext})=0$, and more precisely as $\eps \to 0$:
			\begin{equation}\label{eqn_fluidlimit_near_end}
			\left\{ \begin{array}{rcl}
			\mathscr{X}(t_{ \mathrm{ext}}- \varepsilon) &\sim& 3 \varepsilon^{2}, \\
			\mathscr{Y}(t_{ \mathrm{ext}}- \varepsilon) &\sim& 4 \varepsilon,\\
			\mathscr{Z}(t_{ \mathrm{ext}}- \varepsilon) &\sim& 4 \sqrt{3}  \varepsilon^{3/2}.\end{array}\right.
			\end{equation}
		\end{itemize}
	\end{proposition}

	\begin{figure}[!h]
		\begin{center}
			\includegraphics[width=15cm]{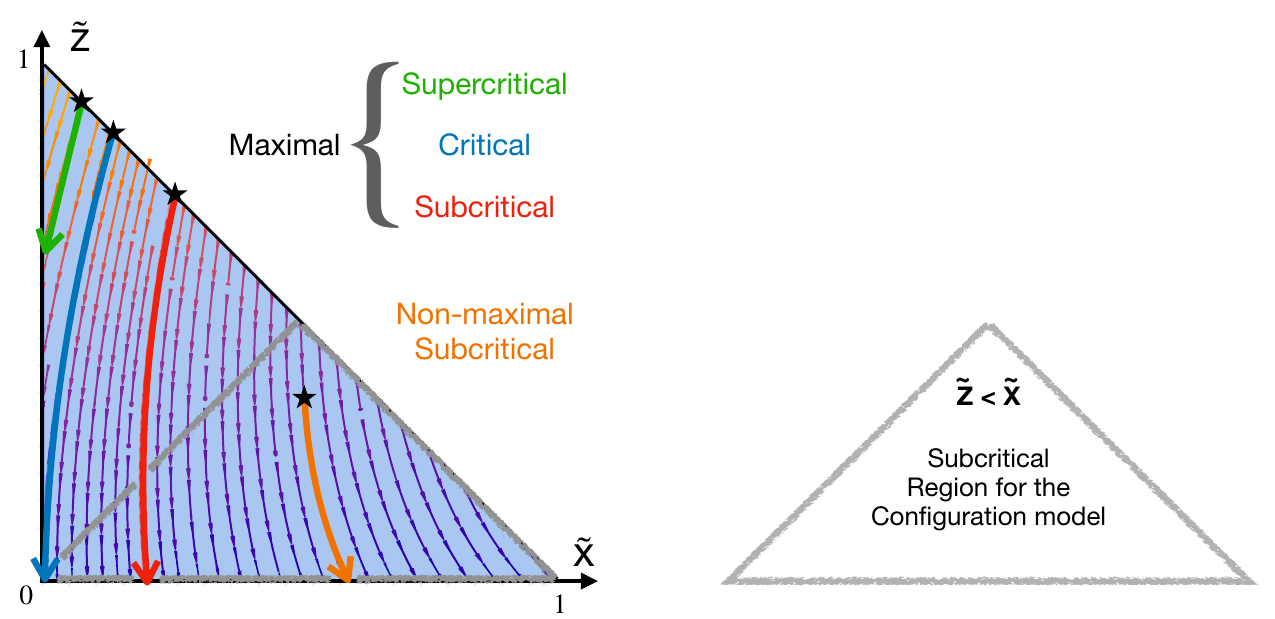}
			\caption{Illustration of the differential system $\tilde{ \mathbf{x}}, \tilde{ \mathbf{z}}$ with the vector field. The maximal solutions start from $ \tilde{ \mathbf{x}}(0)+ \tilde{ \mathbf{z}}(0) =1$. A maximal supercritical (resp.~critical, resp.~subcritical) solution is shown in green (resp.~blue, resp.~red). A non-maximal subcritical solution is displayed in orange. Note that any subcritical solution terminates in the gray region which is subcritical for the configuration model itself.}\label{fig:equadiff_trajectories}
		\end{center}
	\end{figure}

	\begin{proof}
		We will first obtain an explicit (up to time-change) solution to \eqref{eqn_diff_system_primo}. We recall that $ \mathscr{S}=  \mathscr{X}+ \mathscr{Y}+ \mathscr{Z}$ is the fluid limit of the sum process and that $ \mathbf{x}, \mathbf{y}, \mathbf{z}$ are the proportions whose sum is constant and equal to $1$.
 
		Using  $  \mathbf{y}= 1 - \mathbf{x}- \mathbf{z}$, the system \eqref{eqn_diff_system_primo} translates into the following system on $ \mathbf{x},\mathbf{z}$ and $ \mathscr{S}$:
		\begin{eqnarray} \label{eq:sysauto} \left\{ \begin{array}{rcl}
		\mathbf{x}' &=& \frac{1}{ \mathscr{S}} ( \mathbf{x}- \mathbf{z}) \mathbf{z},\\
		\mathbf{z}'&=& \frac{1}{ \mathscr{S}} (-2 + \mathbf{x} - \mathbf{z})  \mathbf{z},\\
		\mathscr{S}'&=& 2 (-2 + \mathbf{x}- \mathbf{z}),
		\end{array}\right.  \end{eqnarray}
		where again $ \mathscr{S}(0)=1$ and $\mathbf{x}(0), \mathbf{z}(0) \geq 0$ satisfy $\mathbf{x}(0)+\mathbf{z}(0) \leq 1$.
		
		In order to get rid of $\mathscr{S}$ in this system, we perform a time change: for $t \in [0,t_{\ext}]$, we write
		\[ \gamma(t)=\int_0^t \frac{\mathrm{d}s}{\mathscr{S}(s)} \in [0,+\infty].\]
		We also define the functions $\tilde{ \mathbf{x}}, \tilde{ \mathbf{y}}, \tilde{ \mathbf{z}}$ on $\left[ 0, u_{\ext}\right]$, with $u_{\ext}=\int_0^{t_{\ext}} \frac{\mathrm{d}s}{\mathscr{S}(s)}$, by $\tilde{ \mathbf{x}}(u)=\mathbf{x} \left( \gamma^{-1}(u) \right)$ and $\tilde{ \mathbf{z}}(u)=\mathbf{z} \left( \gamma^{-1}(u) \right)$. We obtain the system
		
		\begin{align*}
		\left\{ \begin{array}{rcl}
		\tilde{  \mathbf{x}}' &=&   (\tilde{  \mathbf{x}}-\tilde{  \mathbf{z}})\tilde{  \mathbf{z}},\\
		\tilde{  \mathbf{z}}' &=&   (-2+\tilde{  \mathbf{x}}- \tilde{  \mathbf{z}}) \tilde{  \mathbf{z}}.
		\end{array}\right.
		\end{align*}
		We find solutions to this system as follows: by subtracting the second line to the first one, we have $ \tilde{  \mathbf{x}}'-\tilde{  \mathbf{z}}'=2 \tilde{  \mathbf{z}}$ and the second line implies that $\tilde{  \mathbf{x}}-\tilde{  \mathbf{z}} = \left( \frac{ \tilde{  \mathbf{z}}'}{\tilde{  \mathbf{z}}}+2\right)$. Deriving the second identity and comparing, we deduce the following second-order non-linear one-dimensional differential equation:
		$$ 2 \left(\tilde{ \mathbf{z}}\right)^{3} =  \tilde{\mathbf{z}}'' \tilde{\mathbf{z}} - \left( \tilde{\mathbf{z}}'\right)^{2}.$$
		A complete family of solutions is given by
		\begin{eqnarray} \label{eq:solution}\left\{ \begin{array}{rcl}
		\tilde{ \mathbf{z}}(u) &=&   \displaystyle\frac{b^{2}}{\sinh(b(u+u_0))^{2}},\\
		\tilde{ \mathbf{x}}(u) &=&   \displaystyle \left(\frac{b}{\tanh(b(u+u_0))}-1\right)^{2}+1-b^{2},\\
		\tilde{\mathbf{y}}(u) &=& \displaystyle \frac{-2b^2}{\tanh^2 (b(u+u_0))}+\frac{2b}{\tanh (b(u+u_0))}+2b^2-1,
		\end{array}\right.  \end{eqnarray}
		where $b, u_0 \in \R$. We notice that along these solutions, the quantity $(\tilde{\mathbf{z}}-\tilde{\mathbf{x}})^2-4\tilde{\mathbf{x}}$ is constant, and is equal to $4(b^2-1)$, this quantity is equal to the $\Theta$ defined by \eqref{eq:defTheta}:
		  \begin{eqnarray} \label{eq:conservation1}  (\tilde{\mathbf{z}}-\tilde{\mathbf{x}})^2-4\tilde{\mathbf{x}} \equiv 4(b^2-1) = (p_3- p_1)^2-4 p_1  =  \Theta . \end{eqnarray}
		
		We also notice that $\tilde{\mathbf{y}}$ is always increasing and that $\tilde{\mathbf{y}}<0$ for $u$ small enough, which has no meaning in our context. Therefore, every solution is contained in a \emph{maximal} solution, i.e. a solution where the initial condition $(p_1, p_2, p_3)$ satisfies $p_2=0$. 
		From $\tilde{\mathbf{y}}(0)=0$, for such a maximal solution, we get
		\[ u_0=\frac{1}{2b} \log \left(1 + 2 b + \frac{2 b\sqrt{ ( 4 b^2-1)}}{ 2 b-1}\right)>0, \]
		so $p_1=1-\frac{1}{2}\sqrt{4b^2-1}$ and $p_3=\frac{1}{2}\sqrt{4b^2-1}$.
		
		We now come back to the true (non-necessarily maximal) solutions $(\mathscr{X},\mathscr{Y},\mathscr{Z})$ in each of the three cases of Proposition~\ref{prop:solutions_equadiff}. For this, we need to study the time change $\gamma:[0,t_{\ext}] \to [0,u_{\ext}]$. By definition of $\gamma$ and the third line of~\eqref{eq:sysauto}, for all $t \in [0,t_{\ext})$, we have
		\begin{align*} \left\{\begin{array}{rcl}
		\frac{1}{ \mathscr{S}(t)}&=&   \gamma'(t),\\
		\mathscr{S}'(t) &=&2(-2 + \tilde{ \mathbf{x}}(\gamma(t)) - \tilde{ \mathbf{z}}(\gamma(t)).
		\end{array}\right. \label{eq:dsdt}  \end{align*}
		Multiplying those lines and integrating both sides using the exact expressions of $ \tilde{ \mathbf{x}}$ and $ \tilde{ \mathbf{z}}$, we find $\frac{\mathrm{d}}{\mathrm{d}t} \log \mathscr{S}(t) = -4 \frac{\mathrm{d}}{\mathrm{d}t} \log \left( \sinh(b\cdot (\gamma(t)+u_0))\right)$ so the following quantity is constant:
		\begin{equation}
		\label{eq:integrale1ere} \mathscr{S}(t) \sinh^4 \left( b \cdot (\gamma(t)+u_0) \right) =  \mathscr{S}(t) \left( \frac{b^{2}}{ \mathbf{z}(t)}\right)^{2} = \frac{ b^4}{\tilde{ \mathbf{z}}(0)^2}= \frac{b^4}{p_3^2}.
		\end{equation}
		Note that this last equation, combined with the expression of $\mathscr{S}'(t)$, provides a differential equation satisfied by $\mathscr{S}$, from which we could express $\mathscr{S}$ as the inverse bijection of an explicit function. However, this will not be needed in the proof. Given those findings, the rest of the proof is made of easy calculations. Let us proceed. We refer to Figures \ref{fig:equadiff_graphs} and  \ref{fig:equadiff_trajectories} for visualization of the system in terms of proportions or in ``number of legs". 
		
		\paragraph{Subcritical regime.}
		For $\Theta < 0$, we have $\frac{1}{2}<b<1$. In this case, we observe that $\tilde{\mathbf{x}}(u) \geq 1-b^2$ is bounded away from $0$, so the same is true for $\mathbf{x}(t)$ on $[0,t_{\ext}]$. It follows that $\mathscr{S}(t_{\ext})=\frac{\mathscr{X}(t_{\ext})}{\mathbf{x}(t_{\ext})}=0$. Therefore, by~\eqref{eq:integrale1ere}, we have $\mathbf{z}(t_{\ext})=\sqrt{\left( \frac{4b^2-1}{4}\right) \mathscr{S}(t_{\ext})}=0$. In particular, for $t$ sufficiently close to $t_{\ext}$, we have $\mathbf{z}(t)<\mathbf{x}(t)$, so $\mathscr{Z}(t)<\mathscr{X}(t)$.  Note that this also implies $u_{\ext}=+\infty$.
		
		\paragraph{Supercritical regime.}
		For $\Theta >0$, we have $1<b<\frac{\sqrt{5}}{2}$. In this case, the function $\tilde{\mathbf{x}}$ first hits $0$ at time
		\[ \hat{u}_{\ext}=-u_0+\frac{1}{b}\mathrm{Arccoth} \frac{1+\sqrt{b^2-1}}{b}.\]
		This implies that $u_{\ext} \leq \hat{u}_{\ext}$. We claim that we have equality. Indeed, if this is not the case, we have $\mathbf{x}(t_{\ext})=\tilde{\mathbf{x}}(u_{\ext})>0$, so $\mathscr{S}(t_{\ext})=0$, so~\eqref{eq:integrale1ere} implies $\tilde{\mathbf{z}}(u_{\ext})=0$ with $u_{\ext}<+\infty$, which is not possible given the explicit expression of $\tilde{\mathbf{z}}$. Therefore, we have $\tilde{\mathbf{x}}(u_{\ext})=0$. Using \eqref{eq:conservation1} we can compute \[\mathbf{z}(t_{\ext})=\tilde{\mathbf{z}}(u_{\ext})=2\sqrt{b^2-1} \quad \mbox{and} \quad \mathbf{y}(t_{\ext})=1-2\sqrt{b^2-1}\]
		and finally, using~\eqref{eq:integrale1ere}:
		\[\mathscr{S}(t_{\ext})=\frac{4}{p_3^2} \tilde{\mathbf{z}}(u_{\ext})^2=\frac{4(b^2-1)}{p_3^2},\]
		which, once translated in terms of $\Theta$, gives~\eqref{eq:YZ_extinction_equadiff}.
		
		\paragraph{Critical regime.} For $\Theta =0$, the maximal solution starts from $p_3=\frac{\sqrt{3}}{2}$ and $p_1=1-\frac{\sqrt{3}}{2}$, and we have $b=1$. In particular, using $u_0 \geq 0$, we have $\tilde{\mathbf{x}}(u)>0$ for all $u \geq 0$. By the same argument as in the supercritical regime, this implies $u_{\ext}=+\infty$. Therefore, by the exact expression of $\tilde{\mathbf{x}}, \tilde{\mathbf{y}}, \tilde{\mathbf{z}}$, as $t \to t_{\ext}$, we have $\mathbf{x}(t), \mathbf{z}(t) \to 0$ and $\mathbf{y}(t) \to 1$. Therefore, by~\eqref{eq:integrale1ere} at $t=t_{\ext}$, we have $\mathscr{S}(t_{\ext})=0$, so $\mathscr{Y}(t_{\ext})=\mathscr{Z}(t_{\ext})=0$.
		
		Hence, letting $t \to t_{\ext}$ in the third equation of~\eqref{eq:sysauto}, we have $\mathscr{S}'(t) \to -4$ as $t \to t_{\ext}$, so ${\mathscr{S}(t_{\ext}-\eps) \sim 4\eps}$ as $\eps \to 0$. Injecting this in~\eqref{eq:integrale1ere}, we find $\mathbf{z}(t_{\ext}-\eps) \sim \sqrt{3\eps}$, so $\mathscr{Z}(t_{\ext}-\eps) \sim 4 \sqrt{3} \eps^{3/2}$. Finally, we know  from \eqref{eq:conservation1} that $(\mathbf{z}-\mathbf{x})^2-4\mathbf{x}$ is constant equal to $0$, so $\mathbf{x}(t_{\ext}-\eps) \sim \frac{1}{4} \mathbf{z}(t_{\ext}-\eps)^2 \sim \frac{3}{4} \eps$, which gives the asymptotics for $\mathscr{X}$.
	\end{proof}

	\subsection{Phase transition: proof of Theorem \ref{thm:phasetransition}}
	
	\paragraph{Subcritical regime.}
	We assume that $(p_1, p_2, p_3)$ is subcritical, and consider the associated solution $(\mathscr{X}, \mathscr{Y}, \mathscr{Z})$ to the differential equation. By Proposition~\ref{prop:solutions_equadiff}, let $t_1<t_{\ext}$ be such that $\mathscr{Z}(t_1)<\mathscr{X}(t_1)$. By Proposition~\ref{prop:fluid-limit}, we have
	\[ \frac{1}{n} \left( X^n_{\lfloor t_1 n\rfloor}, Y^n_{\lfloor t_1 n\rfloor}, Z^n_{\lfloor t_1 n\rfloor} \right) \xrightarrow[n \to +\infty]{(\P)} \left( \mathscr{X}(t_1), \mathscr{Y}(t_1), \mathscr{Z}(t_1)\right). \]
	Moreover, by Proposition~\ref{prop:markovexplo}, \rev{conditionally} on $\mathcal{F}_{\lfloor t_1 n\rfloor}$, the remaining graph after $\lfloor t_1 n\rfloor$ steps of the Karp--Sipser algorithm is a configuration model with respectively $X^n_{\lfloor t_1 n\rfloor}$, $Y^n_{\lfloor t_1 n\rfloor}$ and $Z^n_{\lfloor t_1 n\rfloor}$ half-edges belonging to vertices of degree $1$, $2$ and $3$. Since $ n^{-1}Z^n_{\lfloor t_1 n\rfloor} \approx \mathscr{Z}(t_1)<\mathscr{X}(t_1)  \approx n^{-1}X^n_{\lfloor t_1 n\rfloor}$ this is a \emph{subcritical} configuration model (do not confuse with subcriticality in terms of the Karp--Sipser core). In particular, by \cite[Theorem 1.b]{molloy1995critical}  there is a constant $c=c(p_1, p_2, p_3)$ such that with high probability the remaining subgraph after $\lfloor t_1 n\rfloor$ steps has fewer than $c \log(n)$ cycles and all of its connected components have size at most $c \log(n)$. On the other hand, by construction, the Karp--Sipser core is included in the union of all the cycles of $G^n_{[t_1 n]}$, so it has size $O_{ \mathbb{P}}( \log^2 n)$.
	
	\begin{remark}[True size of the subcritical KS-core] The above bound $ O_{\mathbb{P}}(\log^2 n)$ for the size of the subcritical Karp--Sipser core is very crude towards the end of the proof. We expect the actual order of magnitude of the KS-core to be $O_{ \mathbb{P}}(1)$ as in the Erd{\H{o}}s--R\'enyi case \cite{aronson1998maximum}.
	\end{remark}

	\paragraph{Critical and supercritical regime.} 
In this case, combining Proposition \ref{prop:fluid-limit} and our explicit computations of the solutions, we obtain that $ ({X}^n/ S^n, {Y}^n /S^n, Z^n/ S^n, n^{-1}\cdot S^n) (\theta^n)$ converges to $$( \mathbf{x} (t_{ \ext}) , \mathbf{y} (t_{\ext}), \mathbf{z} (t_{\ext}),  \mathscr{S} (t_{\ext})) = \left(0, 1 - 2 \sqrt{b^2-1}, 2 \sqrt{b^2-1}, \frac{16(b^2-1)}{4b^2-1}\right).$$ 
	In particular the number of half-edges of the Karp-Sipser core is equal to $S^n_{\theta^n} = Y^n_{\theta^n}+ Z^n_{\theta^n}$, so it is asymptotically $o_{ \mathbb{P}} (n)$ if $b=1$ (critical case). If $b>1$, it is linear in $n$, which concludes the proof of Theorem~\ref{thm:phasetransition} after a quick computation.

\section{Analysis of the critical case}

In this section, we shall prove our main result Theorem \ref{thm:maincritical}. In the rest of the paper, we shall thus suppose that the initial conditions \eqref{eq:convcrit} are in force. Let us first explain the heuristics to help the reader follow the proof. We refer to Figure \ref{fig:explained} for an illustration.

\begin{figure}[!h]
 \begin{center}
 \includegraphics[width=13cm]{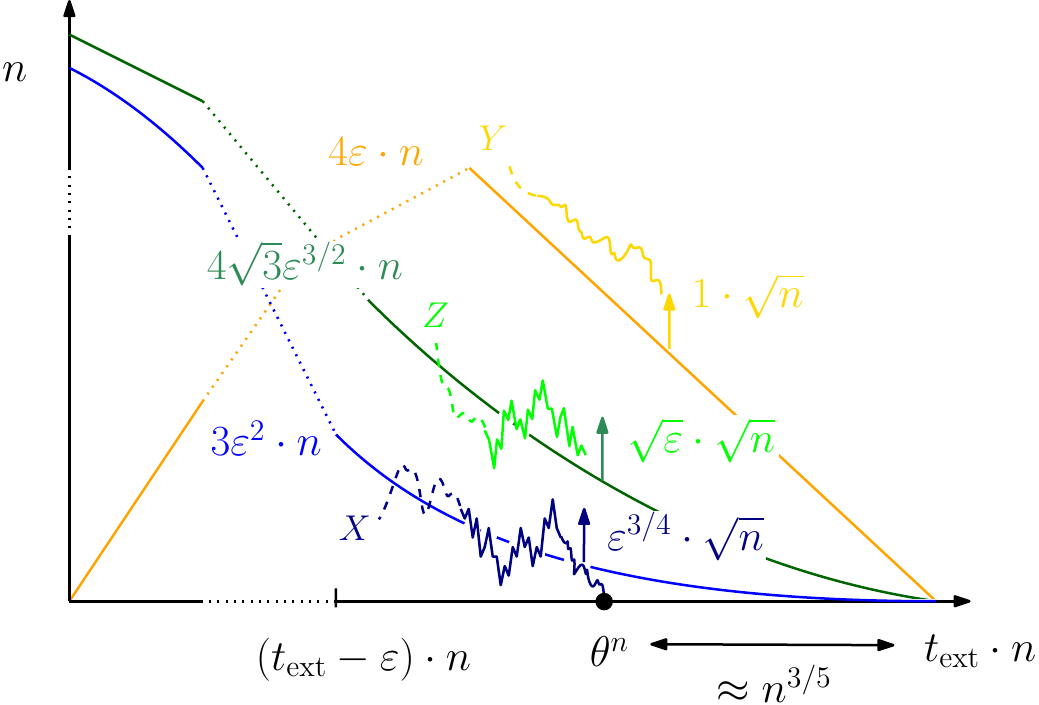}
 \caption{Heuristics for the proof of Theorem \ref{thm:maincritical}. The variations of the processes $(X,Y,Z)$ around its deterministic fluid limit when $ k = (t_{ \mathrm{ext}}- \epsb_k)n$ 	are displayed above. In particular, in the case of $X$, the number of degree $1$ vertices, those variations may cause $X$ to touch $0$ when $ \epsb_k  \approx n^{{-2/5}}$ so that there are $  \epsb_k n \approx n^{3/5}$ vertices of degree $2$ and $ \epsb_k^{{3/2}} n \approx n^{{2/5}}$ vertices of degree $3$ remaining in the graph. \label{fig:explained}}
 \end{center}
 \end{figure}

We have seen above that in the critical regime, the asymptotic size of the Karp-Sipser core is $ o_ \mathbb{P} (n)$ and that almost all vertices have degree 2 (i.e.\ with density $1$ since $ \mathbf{y}(t_ \mathrm{ext})=1$). Recall that the process stops at time $$ \theta^n = \inf \{ k \geq 0 :  X^{n}_{k} = 0 \},$$ which by Proposition \ref{prop:fluid-limit} is $\approx t_{ \mathrm{ext}} \cdot n$. To analyse this stopping time  and understand the size of the KS-core,  we need to be more precise in the analysis of the fluctuations of the process $ (X^n,Y^n,Z^n)$ around its fluid limit $n \cdot (  \mathscr{X},  \mathscr{Y}, \mathscr{Z})$. 
 To this end, we define the fluctuations processes $(A_k^n,B_k^n,C_k^n)_{0 \leq k \leq \theta^n}$ by  \begin{equation*}
\left\{ \begin{array}{rcl}
X_k^n &=& n\mathscr{X}\left( \frac{k}{n}\right) + A_k^n \\
Y_k^n &=& n\mathscr{Y}\left( \frac{k}{n}\right) + B_k^n \\
Z_k^n &=& n\mathscr{Z}\left( \frac{k}{n}\right) + C_k^n\end{array}\right.
\end{equation*}
To simplify notation, the $n$ in the exponent will be implicit for the rest of the paper when there is no ambiguity, even if we will often look at the asymptotic as $n$ goes to infinity.

When we are sufficiently far from the end of the process, i.e. when $ k \approx tn$ for $0 \leq t < t_ \mathrm{ext}$ we know from  Proposition \ref{prop:fluid-limit} that $(X,Y,Z)$ is well approximated by $ n \cdot (  \mathscr{X},  \mathscr{Y}, \mathscr{Z})$ and classical results (see Lemma~\ref{lem:enterregion}) will show that the fluctuations $A,B$ and $C$ renormalized by a factor $1/ \sqrt{n}$ converge to Gaussian variables whose variances depend on $t$. To analyse the algorithm towards the end we will use the notation, for $ 0 \leq k \leq (t_ \mathrm{ext} n) \wedge \theta$, 
 \begin{eqnarray} \label{eq:notationreste} \boxed{\epsb_k := t_ \mathrm{ext} - \frac{k}{n}\geq 0 \qquad \mbox{so that} \qquad k = ( t_ \mathrm{ext}  - \epsb_k)n. } \end{eqnarray}
Notice the bold font for $\epsb$ to avoid confusion. Recall from Equation \eqref{eqn_fluidlimit_near_end} that $ \mathscr{X} (k/n), \mathscr{Y}(k/n)$ and $ \mathscr{Z}(k/n)$ are of order respectively $ \epsb_k^{2}, \epsb_k$ and $ \epsb_k^{3/2}$. We will see below that the order of magnitude of $n^{-1/2}\cdot A_{k}, n^{-1/2}\cdot B_{k}$ and $n^{-1/2}\cdot C_{k}$  are respectively $ \epsb_k^{3/4}, 1$ and  $\epsb_k^{{1/2}}$. In particular, the fluctuations $A$ of $X$ become of the same order of magnitude as its deterministic approximation $ n \mathscr{X}$ when $$ n \epsb_k^{2} \approx n \mathscr{X}( t_{ \mathrm{ext}}- \epsb_k) \approx A_{k} \approx  \sqrt{n} \cdot\epsb_k^{{3/4}}  \quad \mbox{ i.e. when }  \quad  \epsb_k \approx n^{-2/5} \iff  n t_{ \mathrm{ext}}-k \approx n^{3/5},$$ and this explains  heuristically why  $ \theta_{n} =  t_{ \mathrm{ext}}n + O(n^{{3/5}})$ and why the size of the Karp-Sipser core is given essentially by $Y_{\theta_{n}} \approx   n^{3/5}$. The rest of this section makes those \rev{heuristics} rigorous and proves our main result Theorem \ref{thm:maincritical}.

We first provide estimations of the conditional expected  drifts and variances of the increments of the fluctuation processes $(A,B,C)$ in Propositions \ref{prop_drift_estimates} and \ref{prop_variance_estimates}. These propositions support the above heuristics and lead us to introduce the renormalized fluctuations processes
\begin{equation*}
\widetilde{A}_k = \displaystyle \frac{ A_k}{\boldsymbol{\varepsilon}_k^{3/4} \sqrt{n}}, \qquad 
\widetilde{B}_k = \displaystyle \frac{ B_k}{  \sqrt{n}}, \quad \mbox{ and } \quad
\widetilde{C}_k = \displaystyle \frac{ C_k}{ \boldsymbol{\varepsilon}_k^{1/2} \sqrt{n}},
\end{equation*}
which, at least heuristically, should be tight in $k$. After that, our proof consists in two main steps. First we will show that with high probability as $n \to \infty$, we can bound --with some $\log$'s-- the process $ ( \widetilde{A}, \widetilde{B}, \widetilde{C})$ up to time $ O(n^{3/5})$ before $ t_ \mathrm{ext} n$, see Proposition \ref{prop:goodregion}. To do so we will extensively use the fact that for $ \widetilde{C}$ and $ \widetilde{A}$, the conditional expected drifts tend ``to pull them back to $0$''  so that the processes remain small over all scales. Finally, in a second step, we will show that when $ k = n t_{ \mathrm{ext}} - t n^{3/5}$ for $x \in \mathbb{R}$ the fluctuation process $ \widetilde{A}$ is  well approximated by a stochastic differential equation, see Proposition \ref{prop:end}. The fluctuations $ \widetilde{B}$ and $ \widetilde{C}$ are, at this scale, still negligible in front of their differential method approximation.
 
\subsection{Drift and variance estimates}\label{sec:driftvariance}
In this section we compute the conditional expected drift and variance of the fluctuations processes $A,B,C$. Recall the very important notation $ \epsb_k$ introduced in \eqref{eq:notationreste}. 
As explained above, it will turn out that $ \theta \equiv \theta^{n}$ is located around $ t_{ \mathrm{ext}}n - O( n^{{3/5}})$ and in the forthcoming Propositions \ref{prop_drift_estimates} and \ref{prop_variance_estimates} we shall allow a little room and only look at times $k < \theta$ such that $ \epsb_{k} \geq n^{-2/5-1/100}$ (and indeed the fraction $1/100$ is somehow arbitrary). We thus put  
 \begin{eqnarray} \label{eq:defthetatilde} \tilde{\theta} = \theta \wedge \left(t_{ \mathrm{ext}}n - n^{3/5-1/100}\right). \end{eqnarray} 
  Recall from above the notation 
$$ \widetilde{A}_k := \frac{X_{k} - n\mathscr{X}\left( \frac{k}{n}\right)}{  \epsb_{k} ^{3/4}\sqrt{n}}  \quad 
 \widetilde{B}_k := \frac{Y_k - n\mathscr{Y}\left( \frac{k}{n}\right)}{  \sqrt{n}}  \quad \mbox{ and } \quad 
  \widetilde{C}_k := \frac{Z_k - n\mathscr{Z}\left( \frac{k}{n}\right)}{ \epsb^{1/2}_{k} \sqrt{n}} .$$
Recall also that $\mathcal{F}_k$ is the $\sigma$-algebra generated by $\left( X_i, Y_i, Z_i \right)_{0 \leq i \leq k}$. We have chosen the normalization so that the processes $ \widetilde{A}_{k}, \widetilde{B}_{k}$ and $\widetilde{C}_{k}$ are of order $1$ and fluctuate at the time-scale  $ \epsb_{k}n$, which is why the conditional expected drift and variances are all of order $ \frac{1}{ \epsb_{k}n}$.

\begin{proposition}[Drift estimates]\label{prop_drift_estimates}
There exists a constant $K>0$ such that for all $\delta>0$, there is $\eta \equiv \eta(\delta)>0$ such that the following holds for $n$ large enough. For any $   (t_{\ext}-\eta) n  \leq k < \tilde{\theta}$,  
 if we have $|\widetilde{A}_k|, |\widetilde{B}_k|, |\widetilde{C}_k|<1000 \log n$ then:
\begin{align}
\left| \E \left[ \Delta \widetilde{A}_k | \mathcal{F}_k \right] +\frac{1}{\epsb_k n} \frac{1}{4} \widetilde{A}_k \right| & \leq \frac{\delta}{\epsb_k n} |\widetilde{A}_k| + \frac{K \epsb_k^{1/4}}{\epsb_k n} \max \left( |\widetilde{B}_k|, |\widetilde{C}_k| \right) + \frac{K}{\epsb_k n} n^{-1/30}, \label{eqn_drift_estimate_A}\\
\left| \E \left[ \Delta \widetilde{B}_k | \mathcal{F}_k \right] \right| & \leq \frac{K}{\epsb_k n} \sqrt{\epsb_k} \max \left( |\widetilde{A}_k|, |\widetilde{B}_k|, |\widetilde{C}_k| \right) + \frac{K}{\epsb_k n} n^{-1/30}, \label{eqn_drift_estimate_B}\\
\left| \E \left[ \Delta \widetilde{C}_k | \mathcal{F}_k \right] -\frac{1}{\epsb_k n} \left( \frac{3\sqrt{3}}{2} \widetilde{B}_k - \widetilde{C}_k \right) \right| & \leq \frac{\delta}{\epsb_k n} \max \left( |\widetilde{B}_k|, |\widetilde{C}_k| \right) + \frac{K}{\epsb_k n} \epsb_k^{3/4} |\widetilde{A}_k| + \frac{K}{\epsb_k n} n^{-1/30}. \label{eqn_drift_estimate_C}
\end{align}
\end{proposition}

\begin{proposition}[Variance estimates]\label{prop_variance_estimates}
There exists a constant $K$ such that for all $\delta>0$, there is $\eta\equiv \eta(\delta)>0$ such that the following holds for $n$ large enough. For any $   (t_{\ext}-\eta) n  \leq k < \tilde{\theta}$,  
if we have $|\widetilde{A}_k|, |\widetilde{B}_k|, |\widetilde{C}_k|<1000 \log n$ then:
\begin{align}
\left| \var \left( \Delta \widetilde{A}_k | \mathcal{F}_k \right) - \frac{2\sqrt{3}}{\epsb_k n} \right| & \leq \frac{\delta}{\epsb_k n} + \frac{K}{\epsb_k n} n^{-1/30} + \frac{K\epsb_k^{1/2}}{\epsb_k n} \widetilde{A}_k^2 + \frac{K}{n} \max \left( \widetilde{B}_k^2, \widetilde{C}_k^2 \right), \label{eqn_variance_estimate_A} \\
\var \left( \Delta \widetilde{A}_k | \mathcal{F}_k \right) & \leq \frac{2 \sqrt{3}+\delta}{\epsb_k n} + \frac{K}{\epsb_k n} n^{-1/30}, \label{eqn_variance_estimate_A_bis} \\
 \var \left(\Delta \widetilde{B}_k | \mathcal{F}_k \right) & \leq \frac{K \epsb_k}{\epsb_k n} \label{eqn_variance_estimate_B} \\
 \var \left( \Delta \widetilde{C}_k | \mathcal{F}_k \right) & \leq \frac{K \epsb_k^{1/2}}{\epsb_k n}. \label{eqn_variance_estimate_C}
\end{align}
\end{proposition}

The proofs of the above two propositions follow by examining precisely the \rev{transition probabilities} of the Markov chain $(X,Y,Z)$ given by Figure \ref{fig:table} and basic (though important) analysis of the behavior of 
the function $\phi$ (defined by \eqref{eqn_diff_system_primo}) and its gradient $\nabla \phi$ near $t_{ \mathrm{ext}}$. Let us start with a deterministic lemma based on \eqref{eqn_fluidlimit_near_end} controlling $X,Y,Z$ from the processes $\widetilde{A},\widetilde{B},\widetilde{C}$:

\begin{lemma}\label{lem_XYZ_not_too_small}
There are absolute constants $C,c>0$ such that if $|\widetilde{A}_k|, |\widetilde{B}_k|, |\widetilde{C}_k|<1000 \log n$ and $X_k>0$ and $ \epsb_{k} \in [n^{-2/5-1/100}, \eta]$, for $n$ large enough we have
\[ X_k \leq C \epsb_k^2 n \times n^{1/100}, \quad Y_k \leq C \epsb_k n, \quad Z_k \leq C \epsb_k^{3/2} n\]
and
\[S_k \geq Y_k \geq c \epsb_k n.\]
\end{lemma}

\begin{proof} Recall the asymptotics \eqref{eqn_fluidlimit_near_end}. We simply write
\[ X_k \leq n \mathscr{X} \left( \frac{k}{n} \right) + \epsb_k^{3/4} \sqrt{n} \widetilde{A}_k \underset{\eqref{eqn_fluidlimit_near_end}}{\leq} C'\epsb_k^2 n + 1000\epsb_k^{3/4} \sqrt{n} \log n.\]
The assumption $k \leq t_{\ext} n-n^{3/5-1/100}$, i.e. $\epsb_k \geq n^{-2/5-1/100}$, implies that the second term is $O \left( \epsb_k^2 n \times n^{1/100} \right)$. The other two upper bounds can be proved in the same way. Finally, we have
\[ S_k \geq Y_k = n \mathscr{Y} (k/n) + \sqrt{n} \widetilde{B}_k \geq c' \epsb_k n-1000 \sqrt{n} \log n,\]
which is enough to prove the lower bound on $S_k$ since $\epsb_k n \geq n^{3/5-1/100}$ is much larger than $\sqrt{n} \log n$ if $n$ is large enough.
\end{proof}

\begin{proof}[Proof of Proposition~\ref{prop_drift_estimates}] Recall the definition of $\phi$ in \eqref{eqn_diff_system_primo} given in terms of proportions, so that using the notation $s=x+y+z$ we have 
\begin{align*}
\phi_X(x,y,z) &= -2\frac{x}{s}-\frac{yz}{s^2}-3\frac{x^2z}{s^3} -2\frac{xy}{s^2}+ \frac{y^2 z}{s^3}-2\frac{xyz}{s^3}-\frac{z^3}{s^3}-4\frac{x z^2}{s^3},\\
\phi_Y(x,y,z) &= 2 \left( 2 \frac{z^3}{s^3}-\frac{xy}{s^2}-2\frac{y^2z}{s^3}-2\frac{xyz}{s^3}-2\frac{y^2}{s^2}+2\frac{xz^2}{s^3} \right),\\
\phi_Z(x,y,z) &= 3 \left( -\frac{yz}{s^2}-\frac{y^2z}{s^3}-4\frac{yz^2}{s^3}-\frac{x^2z}{s^3}-2\frac{xyz}{s^3}-4\frac{xz^2}{s^3}-3\frac{z^3}{s^3} \right),
\end{align*}
and the fluid limit equation is $\mathscr{X}'=\phi_X(\mathscr{X}, \mathscr{Y}, \mathscr{Z})$, and similarly for the two other coordinates.

\textsc{We start with the estimate~\eqref{eqn_drift_estimate_A} on $\widetilde{A}$}. We first decompose the conditional expected drift as follows:
\begin{align*}
\E \left[ \Delta \widetilde{A}_k | \mathcal{F}_k \right] &= \frac{1}{\epsb_{k+1}^{3/4} \sqrt{n}} \E \left[ \Delta A_k | \mathcal{F}_k \right] + \left( \frac{\epsb_k^{3/4}}{\epsb_{k+1}^{3/4}} - 1 \right) \widetilde{A}_k \\
&= \frac{1}{\epsb_{k+1}^{3/4} \sqrt{n}} \left( \E \left[ \Delta X_k | \mathcal{F}_k \right] -n \left( \mathscr{X} \left( \frac{k+1}{n} \right) - \mathscr{X} \left( \frac{k}{n} \right) \right) \right) + \left( \frac{\epsb_k^{3/4}}{\epsb_{k+1}^{3/4}} - 1 \right) \widetilde{A}_k
\end{align*}
Therefore, by decomposing $1/4 = 1 - 3/4 $, we can decompose the left-hand side of~\eqref{eqn_drift_estimate_A} as follows:
\begin{align}
&\nonumber \left| \E \left[ \Delta \widetilde{A}_k | \mathcal{F}_k \right] +\frac{1}{\epsb_k n} \frac{1}{4} \widetilde{A}_k \right|\\
& \leq \left| \left( \frac{\epsb_k^{3/4}}{\epsb_{k+1}^{3/4}} - 1 \right) \widetilde{A}_k - \frac{3}{4} \frac{1}{\epsb_k n} \widetilde{A}_k  \right| \label{eqn_drift_estimate_A_term_1} \\
&+ \frac{1}{\epsb_{k+1}^{3/4} \sqrt{n}} \left| \E \left[ \Delta X_k | \mathcal{F}_k \right] - \phi_X \left( \frac{X_k}{n}, \frac{Y_k}{n}, \frac{Z_k}{n} \right) \right| \label{eqn_drift_estimate_A_term_2}\\
&+ \frac{1}{\epsb_{k+1}^{3/4} \sqrt{n}} \left| \phi_X \left( \frac{X_k}{n}, \frac{Y_k}{n}, \frac{Z_k}{n} \right) - \phi_X \left( (\mathscr{X}, \mathscr{Y}, \mathscr{Z}) \left( \frac{k}{n} \right) \right) - \left( \frac{A_k}{n}, \frac{B_k}{n}, \frac{C_k}{n} \right) \cdot \nabla \phi_X \left( (\mathscr{X}, \mathscr{Y}, \mathscr{Z}) \left( \frac{k}{n} \right) \right) \right| \label{eqn_drift_estimate_A_term_3}\\
&+ \frac{1}{\epsb_{k+1}^{3/4} \sqrt{n}} \left|  \left( \frac{A_k}{n}, \frac{B_k}{n}, \frac{C_k}{n} \right) \cdot \nabla \phi_X \left( (\mathscr{X}, \mathscr{Y}, \mathscr{Z}) \left( \frac{k}{n} \right) \right) + \frac{1}{\epsb_k} \frac{A_{k}}{n} \right| \label{eqn_drift_estimate_A_term_4}\\
&+ \frac{1}{\epsb_{k+1}^{3/4} \sqrt{n}} \left| \phi_X \left( (\mathscr{X}, \mathscr{Y}, \mathscr{Z}) \left( \frac{k}{n} \right) \right) - n \left( \mathscr{X} \left( \frac{k+1}{n} \right) - \mathscr{X} \left( \frac{k}{n} \right) \right) \right|. \label{eqn_drift_estimate_A_term_5}
\end{align}
We will bound each of these five error terms one by one. More precisely, we will prove that the terms~\eqref{eqn_drift_estimate_A_term_1}, \eqref{eqn_drift_estimate_A_term_2}, \eqref{eqn_drift_estimate_A_term_3} and~\eqref{eqn_drift_estimate_A_term_5} are all $O \left( \frac{n^{-1/30}}{\epsb_k n} \right)$, whereas the other terms in~\eqref{eqn_drift_estimate_A} come from~\eqref{eqn_drift_estimate_A_term_4}. We start with \eqref{eqn_drift_estimate_A_term_1}, which is easy. We simply write $\epsb_k=t_{\ext}-\frac{k}{n}$ and $\epsb_{k+1}=t_{\ext}-\frac{k+1}{n}$. This implies $\frac{\epsb_{k+1}}{\epsb_k}=1-\frac{1}{\epsb_{k} n}$, so
\[ \frac{\epsb_k^{3/4}}{\epsb_{k+1}^{3/4}} - 1 = \frac{3}{4} \frac{1}{\epsb_k n} + O \left( \frac{1}{(\epsb_k n)^2} \right), \]
where the constant is absolute. Finally, using $\epsb_k n \geq n^{3/5-1/100}$ we have
\[ \frac{|\widetilde{A}_k|}{(\epsb_k n)^2} \leq \frac{100 n^{-3/5+1/100} \log n}{\epsb_k n},\]
so we can bound~\eqref{eqn_drift_estimate_A_term_1} by $\frac{K}{\epsb_k n} n^{-1/30}$.

We now move on to~\eqref{eqn_drift_estimate_A_term_2}. The drift $\E \left[ \Delta X_k | \mathcal{F}_k \right]$ can be expressed as the sum over all the cases of Figure~\ref{fig:transitions} of the probability of each case multiplied by the variation of $X$ in this case. For example, the probability for the first case is $\frac{X_k-1}{S_k-1}$. Approximating $\E \left[ \Delta X_k | \mathcal{F}_k \right]$ by $\phi_X \left( \frac{X_k}{n},\frac{Y_k}{n}, \frac{Z_k}{n} \right)$ is then equivalent to approximating $\frac{X_k-1}{S_k-1}$ by $\frac{X_k/n}{S_k/n}$, and similarly for all the other terms. But we have
\[ \frac{X_k-1}{S_k-1}=\frac{X_k/n}{S_k/n} \times \frac{1-\frac{1}{X_k}}{1-\frac{1}{S_k}} = \frac{X_k/n}{S_k/n} \left( 1-O \left( \frac{1}{X_k} \right) +O \left( \frac{1}{S_k} \right)\right) = \frac{X_k/n}{S_k/n} + O \left( \frac{1}{S_k} \right), \]
since $S_k \geq X_k$. When we do the same computation for all the cases of Figure~\ref{fig:transitions}, we also get an error $O ( \frac{1}{S_k} )$. Note that for the last three cases on the bottom right of Figure~\ref{fig:transitions}, the probability is already $O ( \frac{1}{S_k} )$, so these cases do not contribute to $\phi_X (x,y,z)$. So we can bound~\eqref{eqn_drift_estimate_A_term_2} by
\[ \frac{1}{\epsb_{k+1}^{3/4} \sqrt{n}}  O\left( \frac{1}{S_k} \right) \underset{ \mathrm{Lem.} \ref{lem_XYZ_not_too_small}}{=} O \left( \frac{1}{\epsb_{k+1}^{3/4} \sqrt{n}} \times \frac{1}{\epsb_k n} \right) \underset{\epsb_{k} \geq n^{- \frac{2}{5}-\frac{1}{100}}}{=} O \left( \frac{1}{\epsb_k n} \times \frac{1}{(n^{- \frac{2}{5}-\frac{1}{100}})^{3/4} \sqrt{n}} \right) = O \left( \frac{n^{-1/30}}{\epsb_k n} \right).  \]

We move on to~\eqref{eqn_drift_estimate_A_term_3}. We want to estimate the error when we do a linear approximation of $\phi_X$ near $(\mathscr{X}, \mathscr{Y}, \mathscr{Z}) \left( \frac{k}{n} \right)$, so we will need to bound the second derivatives of $\phi_X$ near this point. More precisely, we write $(v_1,v_2,v_3)=\left( \frac{A_k}{n}, \frac{B_k}{n}, \frac{C_k}{n} \right)$. By the Taylor-Lagrange formula we can bound~\eqref{eqn_drift_estimate_A_term_3} by
\begin{equation}\label{eqn_bd_second_derivative_A}
\frac{1}{\epsb_{k+1}^{3/4}\sqrt{n}} \sum_{1 \leq i,j \leq 3} |v_i| \times |v_j| \times \max_{\substack{|u_1-\mathscr{X}(k/n)|\leq |v_1|\\ |u_2-\mathscr{Y}(k/n)|\leq |v_2|\\ |u_3-\mathscr{X}(k/n)|\leq |v_3|}} \left| \frac{\partial^2 \phi_X}{\partial x_i \partial x_j} (u_1,u_2,u_3) \right|.
\end{equation}
By the assumptions of the proposition, we have the bounds:
\begin{equation}\label{eqn_bd_v1v2v3}
|v_1| \leq 1000 \, \epsb_k^{3/4} \frac{\log n}{\sqrt{n}}, \quad |v_2| \leq 1000 \frac{\log n}{\sqrt{n}}, \quad |v_3| \leq 1000 \, \epsb_k^{1/2} \frac{\log n}{\sqrt{n}}.
\end{equation}
On the other hand, we can compute the second order derivatives of $\phi_X$, which are of the form $\frac{P(x,y,z)}{(x+y+z)^4}$ for some polynomial $P$. By Lemma~\ref{lem_XYZ_not_too_small}, we know that $u_1$, $u_2$ and $u_3$ are respectively $O \left( \epsb_k^2 n^{1/100} \right)$, $O \left( \epsb_k n^{1/100} \right)$ and $O \left( \epsb_k^{3/2} n^{1/100} \right)$, and the sum $u_1+u_2+u_3$ is of order $\epsb_k$.  Hence, we can consider the term with the highest order in the numerator. For example, we find
\[\frac{\partial^2 \phi_X}{\partial x_1^2}=\frac{12u_2^2+28u_2u_3+10u_3^2 }{(u_1+u_2+u_3)^4}, \]
and the highest order term in the numerator is $u_2^2=O \left( \epsb_k^2 n^{1/50} \right)$. On the other hand, the denominator is of order $\epsb_k^4$, so we get
\[\frac{\partial^2 \phi_X}{\partial x_1^2} (u_1, u_2, u_3)= O \left( \epsb_k^{-2} n^{1/50} \right).\]
The bounds on $\frac{\partial^2 \phi_X}{\partial x_i \partial x_j} (u_1,u_2,u_3)$ that we obtain for all second-order partial derivatives  are summarized in the following table:

\begin{center}
\begin{tabular}{|c|c|c|c|}
	\hline
	$i \backslash j$ & 1 & 2 & 3 \\
	\hline
	1 & $O \left( \epsb_k^{-2} n^{1/50} \right)$ & $O \left( \epsb_k^{-3/2} n^{1/50} \right)$ & $O \left( \epsb_k^{-2} n^{1/50} \right)$ \\
	\hline
	2 & $O \left( \epsb_k^{-3/2} n^{1/50} \right)$ & $O \left( \epsb_k^{-1} n^{1/50} \right)$ & $O \left( \epsb_k^{-3/2} n^{1/50} \right)$ \\
	\hline
	3 & $O \left( \epsb_k^{-2} n^{1/50} \right)$ & $O \left( \epsb_k^{-3/2} n^{1/50} \right)$ & $O \left( \epsb_k^{-2} n^{1/50} \right)$ \\
	\hline
\end{tabular}
\end{center}

Combining this with~\eqref{eqn_bd_v1v2v3}, we find that each term of~\eqref{eqn_bd_second_derivative_A} is $$O \left( \frac{\epsb_k^{-1} n^{1/50} \log^2 n}{\epsb_{k+1}^{3/4} n \sqrt{n}} \right)=O \left( \frac{1}{\epsb_k n} \times \frac{n^{1/50} \log^2 n}{\epsb_{k+1}^{3/4} \sqrt{n}} \right) \underset{2 \epsb_{k+1} \geq n^{-2/5-1/100}}{=}O \left( \frac{n^{-1/30}}{\epsb_k n} \right),$$ which bounds~\eqref{eqn_drift_estimate_A_term_3}. Note that it was necessary to handle one by one the terms of~\eqref{eqn_bd_second_derivative_A} and not to bound everything crudely by $\| v \|^2 \times \| D^2 \phi_X \|$ (we would have obtained an additional factor $\epsb_k^{-1}$, which is too large).

Let us now bound~\eqref{eqn_drift_estimate_A_term_4}. We first compute the gradient of $\nabla \phi_X$:
\begin{equation}\label{eqn_gradient_phiX}
\nabla \phi_X (x,y,z) = \frac{1}{(x+y+z)^3} \left( -4y^2-9yz+xz-3z^2, 4xy+6xz+2z^2, -x^2-2yz+3xy+3xz \right).
\end{equation}
On the other hand, by~\eqref{eqn_fluidlimit_near_end}, when $ \varepsilon \to 0$, we have
\[ \mathscr{X}(t_{\ext}- \varepsilon) \sim 3  \varepsilon^2, \quad \mathscr{Y}(t_{\ext}- \varepsilon) \sim 4  \varepsilon \quad, \mathscr{Z}( t_{\ext}- \varepsilon) \sim 4\sqrt{3}  \varepsilon^{3/2}. \]
Therefore, we can replace $(x,y,z)$ in~\eqref{eqn_gradient_phiX} by $\left( \mathscr{X}(t), \mathscr{Y}(t), \mathscr{Z}(t) \right)$ and let $t \to t_{\ext}$. We find that there are constants $K, \eta>0$ such that, for any $0<  \varepsilon <  \eta$, we have:
\begin{align*}
\left| \frac{\partial \phi_X}{\partial x} \left(\left( \mathscr{X}, \mathscr{Y}, \mathscr{Z} \right)(t_{\ext}- \varepsilon)\right)-\frac{1}{ \varepsilon} \right| & \leq \frac{\delta}{ \varepsilon},\\
\left| \frac{\partial \phi_X}{\partial y} \left(\left( \mathscr{X}, \mathscr{Y}, \mathscr{Z} \right)(t_{\ext}- \varepsilon)\right) \right| & \leq K,\\
\left| \frac{\partial \phi_X}{\partial y} \left(\left( \mathscr{X}, \mathscr{Y}, \mathscr{Z} \right)(t_{\ext}- \varepsilon)\right) \right| & \leq \frac{K}{ \varepsilon^{1/2}}.
\end{align*}
This is the value of $\eta$ that we take in Proposition~\ref{prop_drift_estimates}. We can now replace $ \varepsilon$ by $ \epsb_{k} \in (0, \eta)$ and we obtain the following bound on~\eqref{eqn_drift_estimate_A_term_4}:
\begin{align*}
& \frac{\delta}{\epsb_k n} \times \frac{1}{\epsb_{k+1}^{3/4} \sqrt{n}} | A_k | + \frac{K}{n} \times \frac{1}{\epsb_{k+1}^{3/4} \sqrt{n}} |B_k| + \frac{K}{\epsb_{k}^{1/2} n} \times \frac{1}{\epsb_{k+1}^{3/4} \sqrt{n}} |C_k|\\
&= \frac{\delta}{\epsb_k n} |\widetilde{A}_k| + \frac{K}{\epsb_k n} \epsb_k^{1/4} |\widetilde{B}_k| + \frac{K}{\epsb_k n} \epsb_k^{1/4} |\widetilde{C}_k|\\
& \leq \frac{\delta}{\epsb_k n} |\widetilde{A}_k| + \frac{1000 K}{\epsb_k n} \epsb_k^{1/4} \log n.
\end{align*}

We finally treat the term~\eqref{eqn_drift_estimate_A_term_5}. We recall that $\mathscr{X}$ solves the equation $\mathscr{X}'=\phi_X (\mathscr{X}, \mathscr{Y}, \mathscr{Z})$, so this is just a linear approximation, so we will need to bound the second derivative $\mathscr{X}''$. More precisely,~\eqref{eqn_drift_estimate_A_term_5} is bounded by
\begin{equation}\label{eqn_drift_term_5_bound}
\frac{n}{\epsb_{k+1}^{3/4} \sqrt{n}} \times \left( \frac{1}{n} \right)^2 \times \max_{\left[ \frac{k}{n}, \frac{k+1}{n} \right]} \left| \mathscr{X}'' \right|.
\end{equation}
Moreover, by differentiating $\mathscr{X}'=\phi_X \left( \mathscr{X},\mathscr{Y},\mathscr{Z}\right)$, we have
\begin{align*}
\mathscr{X}''(t) &= \left( \phi_X \frac{\partial \phi_X}{\partial x} + \phi_Y \frac{\partial \phi_X}{\partial y} + \phi_Z \frac{\partial \phi_X}{\partial z} \right) \left(  \mathscr{X}(t), \mathscr{Y}(t), \mathscr{Z}(t) \right) \\
&= \frac{\mathscr{Z}(t)}{\mathscr{S}(t)^4} \left( \mathscr{X}(t)^2 -2 \mathscr{X}(t) \mathscr{Y}(t) +8 \mathscr{Y}(t) \mathscr{Z}(t) +11 \mathscr{Z}(t)^2 \right).
\end{align*}
This is a continuous function of $t$ on $[0,t_{\ext})$. Moreover, by~\eqref{eqn_fluidlimit_near_end}, we have
\[ \mathscr{X}(t_{\ext}-\eps) \sim_{\eps \to 0} \frac{4 \sqrt{3} \eps^{3/2}}{(4\eps)^4} \times 8 \times 4 \eps \times 4 \sqrt{3} \eps^{3/2} = 6,\]
so $\mathscr{X}''$ is bounded by a constant $K$. Plugging this into~\eqref{eqn_drift_term_5_bound}, we can bound~\eqref{eqn_drift_estimate_A_term_5} by $\frac{K}{\epsb_{k+1}^{3/4} n^{3/2}}=O \left( \frac{n^{-1/30}}{\epsb_k n} \right)$.

\textsc{We now move on to the estimates~\eqref{eqn_drift_estimate_B} and~\eqref{eqn_drift_estimate_C}}. Since the  proof is similar, we will not do it in full details and only stress the differences with the proof of~\eqref{eqn_drift_estimate_A}. The decomposition of the error into five terms is the same with the following modifications:
\begin{itemize}
	\item the first term~\eqref{eqn_drift_estimate_A_term_1} becomes $\left| \left( \frac{\epsb_k^{1/2}}{\epsb_{k+1}^{1/2}} - 1 \right) \widetilde{C}_k - \frac{1}{2} \frac{1}{\epsb_k n} \widetilde{C}_k  \right|$ for $\widetilde{C}$, and disappears completely for $\widetilde{B}$;
	\item in the terms~\eqref{eqn_drift_estimate_A_term_2}, \eqref{eqn_drift_estimate_A_term_3}, \eqref{eqn_drift_estimate_A_term_4} and~\eqref{eqn_drift_estimate_A_term_5}, the factors $\frac{1}{\epsb_{k+1}^{3/4}\sqrt{n}}$ become $\frac{1}{\sqrt{n}}$ for $\widetilde{B}$ and $\frac{1}{\epsb_{k+1}^{1/2}\sqrt{n}}$ for $\widetilde{C}$;
	\item in the fourth term~\eqref{eqn_drift_estimate_A_term_4}, the drift $\frac{1}{\epsb_k n} \widetilde{A}_k$ becomes $0$ for $\widetilde{B}$ and $\frac{3\sqrt{3}}{2} \widetilde{B}_k - \frac{3}{2} \widetilde{C}_k$ of $\widetilde{C}$. 
\end{itemize}
The first and second term can then be bounded by $O \left( \frac{n^{-1/30}}{\epsb_k n} \right)$ in the exact same way as for $\widetilde{A}$ (this bound actually becomes cruder for~\eqref{eqn_drift_estimate_A_term_2}, since now the factor $\epsb_{k+1}^{3/4}$ in the denominator disappears or become larger).

The bound on the fifth term~\eqref{eqn_drift_estimate_A_term_5} is also very similar: we now have
\[ \mathscr{Y}''(t)=-2\frac{\mathscr{Z}}{\mathscr{S}^4} \left( \mathscr{X} \mathscr{Y} + 4 \mathscr{Y}^2 + 8 \mathscr{X}\mathscr{Z} + 21\mathscr{Y}\mathscr{Z} + 20 \mathscr{Z}^2 \right)(t) = O \left( \left(t_{\ext}-t \right)^{-1/2} \right)\]
\[ \mathscr{Z}''(t) = 3\frac{\mathscr{Z}}{\mathscr{S}^4} \left( \mathscr{X}^2 + 4\mathscr{X}\mathscr{Y} + 4\mathscr{Y}^2 + 8\mathscr{X}\mathscr{Z} + 14\mathscr{Y}\mathscr{Z} + 11\mathscr{Z}^2 \right)(t) = O \left( \left(t_{\ext}-t \right)^{-1/2} \right). \]
Therefore, the analog of~\eqref{eqn_drift_estimate_A_term_5} for $\widetilde{B}$ (resp. $\widetilde{C}$) is $O \left( \frac{1}{\sqrt{n}} \times n \times \left( \frac{1}{n} \right)^2 \times \epsb_k^{-1/2} \right)=O \left( \frac{\epsb_k^{-1/2}}{n^{3/2}} \right)$ (resp. $O \left( \frac{\epsb_k^{-1}}{n^{3/2}} \right)$). In both cases, this is $O \left( \frac{n^{-1/30}}{\epsb_k n} \right)$.

The analog of the third term~\eqref{eqn_drift_estimate_A_term_3} is still very similar, but requires to be more careful. Indeed~\eqref{eqn_bd_second_derivative_A} becomes respectively
\begin{equation}\label{eqn_bd_second_derivative_B}
\frac{1}{\epsb_{k+1}^{1/2}\sqrt{n}} \sum_{1 \leq i,j \leq 3} |v_i| \times |v_j| \times \max_{\substack{|u_1-\mathscr{X}(k/n)|\leq |v_1|\\ |u_2-\mathscr{Y}(k/n)|\leq |v_2|\\ |u_3-\mathscr{Z}(k/n)|\leq |v_3|}} \left| \frac{\partial^2 \phi_Y}{\partial x_i \partial x_j} (u_1,u_2,u_3) \right|.
\end{equation}
and
\begin{equation}\label{eqn_bd_second_derivative_C}
\frac{1}{\sqrt{n}} \sum_{1 \leq i,j \leq 3} |v_i| \times |v_j| \times \max_{\substack{|u_1-\mathscr{X}(k/n)|\leq |v_1|\\ |u_2-\mathscr{Y}(k/n)|\leq |v_2|\\ |u_3-\mathscr{Z}(k/n)|\leq |v_3|}} \left| \frac{\partial^2 \phi_Z}{\partial x_i \partial x_j} (u_1,u_2,u_3) \right|.
\end{equation}
for $\widetilde{B}$ and $\widetilde{C}$. Moreover, when we compute the second order partial derivatives $\frac{\partial^2 \phi_Y}{\partial x_i \partial x_j}$ and $\frac{\partial^2 \phi_Y}{\partial x_i \partial x_j}$, we get respectively the following tables:
\begin{center}
	\begin{tabular}{|c|c|c|c|}
		\hline
		$i \backslash j$ & 1 & 2 & 3 \\
		\hline
		1 & $O \left( \epsb_k^{-2} n^{1/50} \right)$ & $O \left( \epsb_k^{-2} n^{1/50} \right)$ & $O \left( \epsb_k^{-2} n^{1/50} \right)$ \\
		\hline
		2 & $O \left( \epsb_k^{-2} n^{1/50} \right)$ & $O \left( \epsb_k^{-3/2} n^{1/50} \right)$ & $O \left( \epsb_k^{-2} n^{1/50} \right)$ \\
		\hline
		3 & $O \left( \epsb_k^{-2} n^{1/50} \right)$ & $O \left( \epsb_k^{-2} n^{1/50} \right)$ & $O \left( \epsb_k^{-3/2} n^{1/50} \right)$ \\
		\hline
	\end{tabular}

	\begin{center}
		\begin{tabular}{|c|c|c|c|}
			\hline
			$i \backslash j$ & 1 & 2 & 3 \\
			\hline
			1 & $O \left( \epsb_k^{-3/2} n^{1/50} \right)$ & $O \left( \epsb_k^{-3/2} n^{1/50} \right)$ & $O \left( \epsb_k^{-2} n^{1/50} \right)$ \\
			\hline
			2 & $O \left( \epsb_k^{-3/2} n^{1/50} \right)$ & $O \left( \epsb_k^{-3/2} n^{1/50} \right)$ & $O \left( \epsb_k^{-2} n^{1/50} \right)$ \\
			\hline
			3 & $O \left( \epsb_k^{-2} n^{1/50} \right)$ & $O \left( \epsb_k^{-2} n^{1/50} \right)$ & $O \left( \epsb_k^{-2} n^{1/50} \right)$ \\
			\hline
		\end{tabular}
	\end{center}
\end{center}
In both cases, using~\eqref{eqn_bd_v1v2v3}, we find that each term of~\eqref{eqn_bd_second_derivative_B} or~\eqref{eqn_bd_second_derivative_C} is $$O \left( \frac{\epsb_k^{-3/2} n^{1/50} \log^2 n}{\epsb_{k+1}^{1/2} n^{3/2}} \right) = O\left( \frac{1}{\epsb_k n} \times \frac{n^{1/50} \log^2 n}{\epsb_k \sqrt{n}} \right) \underset{\epsb_k \geq n^{- \frac{2}{5}- \frac{1}{100}}}{=} O \left( \frac{n^{-1/30}}{\epsb_k n} \right).$$

Finally, to handle the analog of the fourth term~\eqref{eqn_drift_estimate_A_term_4}, we just need to compute the gradients of $\phi_Y$ and $\phi_Z$:
\[ \nabla \phi_Y(x,y,z) = \frac{1}{(x+y+z)^3} \left( 2xy+6y^2+6yz-8z^2, -2x^2-6xy-6xz+4yz+12z^2, 8xz+4y^2+12yz \right), \]
\[ \nabla \phi_Z(x,y,z) = \frac{1}{(x+y+z)^3} \left( 3xz+9yz+15z^2, 6yz+12z^2, -3x^2-9xy-15xz-6y^2-12yz \right).\]
As in the first case, we can now replace $(x,y,z)$ by $\left( \mathscr{X}(t), \mathscr{Y}(t), \mathscr{Z}(t) \right)$, use~\eqref{eqn_fluidlimit_near_end} and identify the highest order terms in $t_{\ext}-t$. We find that there is a constant $K$ such that for all $0 \leq t < t_{\ext}$:
\begin{align*}
\left| \frac{\partial \phi_Y}{\partial x} \left( \mathscr{X}(t), \mathscr{Y}(t), \mathscr{Z}(t) \right) \right| &\leq \frac{K}{t_{\ext}-t},\\
\left| \frac{\partial \phi_Y}{\partial y} \left( \mathscr{X}(t), \mathscr{Y}(t), \mathscr{Z}(t) \right) \right| &\leq \frac{K}{(t_{\ext}-t)^{1/2}},\\
\left| \frac{\partial \phi_Y}{\partial z} \left( \mathscr{X}(t), \mathscr{Y}(t), \mathscr{Z}(t) \right) \right| &\leq \frac{K}{t_{\ext}-t},\\
\left| \frac{\partial \phi_Z}{\partial x} \left( \mathscr{X}(t), \mathscr{Y}(t), \mathscr{Z}(t) \right) \right| &\leq \frac{K}{(t_{\ext}-t)^{1/2}}.
\end{align*}
Moreover, there is $\eta>0$ (depending on $\delta$) such that, if $t_{\ext}-\eta \leq t < t_{\ext}$, then
\begin{align*}
\left| \frac{\partial \phi_Z}{\partial y} \left( \mathscr{X}(t), \mathscr{Y}(t), \mathscr{Z}(t) \right) - \frac{3\sqrt{3}}{2} \frac{1}{(t_{\ext}-t)^{1/2}} \right| &\leq \frac{\delta}{(t_{\ext}-t)^{1/2}},\\
\left| \frac{\partial \phi_Z}{\partial y} \left( \mathscr{X}(t), \mathscr{Y}(t), \mathscr{Z}(t) \right) + \frac{3}{2} \frac{1}{t_{\ext}-t} \right| &\leq \frac{\delta}{t_{\ext}-t}.
\end{align*}
From here, taking $t=\frac{k}{n}$ and replacing $\left( A_k, B_k, C_k \right)$ by $\left( \epsb_k^{3/4} \sqrt{n} \widetilde{A}_k, \sqrt{n} \widetilde{B}_k, \epsb_k^{1/2} \sqrt{n} \widetilde{C}_k \right)$, we easily obtain the claimed bound on~\eqref{eqn_drift_estimate_A_term_4}.
\end{proof}

\begin{proof}[Proof of Proposition~\ref{prop_variance_estimates}]
Just like in the proof of Proposition~\ref{prop_drift_estimates}, we first introduce the following functions (again with the notation $s=x+y+z$):
\begin{align*}
\psi_X(x,y,z) &= 4\frac{x}{s}+4\frac{xy}{s^2}+\frac{yz}{s^2}+\frac{y^2z}{s^3}+9\frac{x^2 z}{s^3}+2\frac{xyz}{s^3}+2\frac{xz^2}{s^3}+\frac{z^3}{s^3},\\
\psi_Y(x,y,z) &= \frac{xy}{s^2}+4\frac{y^2}{s^2}+4\frac{y^2z}{s^3}+2\frac{xyz}{s^3}+2\frac{xz^2}{s^3}+4\frac{z^3}{s^3},\\
\psi_Z(x,y,z) &= \frac{yz}{s^2}+\frac{y^2z}{s^3}+8\frac{yz^2}{s^3}+\frac{x^2 z}{s^3}+2\frac{xyz}{s^3}+8\frac{xz^2}{s^3}+9\frac{z^3}{s^3}.
\end{align*}
These functions are respectively the fluid limit approximations of $ \mathbb{E}[(\Delta X_{k})^{2}\mid \mathcal{F}_{k}]$, $\mathbb{E}[(\Delta Y_{k})^{2}\mid \mathcal{F}_{k}]$ and $\mathbb{E}[(\Delta Z_{k})^{2}\mid \mathcal{F}_{k}]$ and can be computed from the \rev{transition probabilities} given in Figure \ref{fig:transitions} as before. \\	
\noindent \textsc{Variance of $\widetilde{A}$}.	Let us start by establishing \eqref{eqn_variance_estimate_A}. We first note that, since adding a function of $A_k$ does not change the conditional variance on $\mathcal{F}_k$, we have
\[ \var \left( \Delta \widetilde{A}_k | \mathcal{F}_k \right) = \var \left( \Delta \widetilde{A}_k + \left( \frac{1}{\epsb_k^{3/4} \sqrt{n}}-\frac{1}{\epsb_{k+1}^{3/4} \sqrt{n}} \right) A_k | \mathcal{F}_k \right)= \frac{1}{\epsb_{k+1}^{3/2} n} \var \left( \Delta A_k | \mathcal{F}_k \right) = \frac{1}{\epsb_{k+1}^{3/2} n} \var \left( \Delta X_k | \mathcal{F}_k \right).\]
Therefore, we can write
\[\var \left( \Delta \widetilde{A}_k | \mathcal{F}_k \right) = \frac{1}{\epsb_{k+1}^{3/2} n} \E \left[ (\Delta X_k)^2 | \mathcal{F}_k \right] - \frac{1}{\epsb_{k+1}^{3/2} n} \E \left[ \Delta X_k | \mathcal{F}_k \right]^2, \]
so
\begin{align} \label{eqn_decomposition_variance_A}
\left| \var \left( \Delta \widetilde{A}_k | \mathcal{F}_k \right) - \frac{2\sqrt{3}}{\epsb_{k} n} \right| &\leq \frac{1}{\epsb_{k+1}^{3/2} n} \E \left[ \Delta X_k | \mathcal{F}_k \right]^2\\ \nonumber
&+ \frac{1}{\epsb_{k+1}^{3/2}n} \left| \E \left[ (\Delta X_k)^2 | \mathcal{F}_k \right] - \psi_X \left( \frac{X_k}{n}, \frac{Y_k}{n}, \frac{Z_k}{n} \right) \right|\\ \nonumber
&+ \frac{1}{\epsb_{k+1}^{3/2}n} \left| \psi_X \left( \frac{X_k}{n}, \frac{Y_k}{n}, \frac{Z_k}{n} \right) -  \psi_X \left( \mathscr{X} \left( \frac{k}{n} \right), \mathscr{Y} \left( \frac{k}{n} \right), \mathscr{Z} \left( \frac{k}{n} \right) \right) \right|\\
&+ \frac{1}{\epsb_{k+1}^{3/2}n} \left| \psi_X \left( \mathscr{X} \left( \frac{k}{n} \right), \mathscr{Y} \left( \frac{k}{n} \right), \mathscr{Z} \left( \frac{k}{n} \right) \right) -  {\left({2\sqrt{3} \sqrt{\epsb_{k}}}\right) }\right|  \nonumber \\ 
& {+ \frac{1}{\epsb_{k+1}^{3/2}n} \left| \left(2\sqrt{3} \sqrt{\epsb_{k}}\right)- \left({2\sqrt{3} \sqrt{\epsb_{k+1}}}\right) \right|} \nonumber
\end{align}
The first term can be bounded by using Proposition~\ref{prop_drift_estimates} and $ \epsb_{k} \geq n^{-2/5-1/100}$.
Moreover, by the exact same argument as for term~\eqref{eqn_drift_estimate_A_term_2} in the proof of Proposition~\ref{prop_drift_estimates}, the second term is
\[ O  \left( \frac{1}{\epsb_{k+1}^{3/2} n} \times \frac{1}{S_k} \right) \underset{ \mathrm{Lem.} \ref{lem_XYZ_not_too_small}}= O \left( \frac{1}{\epsb_k^{5/2} n^2} \right) \underset{\epsb_k \geq n^{- \frac{2}{5}- \frac{1}{100}}}{=} O \left( \frac{n^{-1/30}}{\epsb_k n} \right).\]

We now bound the third term of~\eqref{eqn_decomposition_variance_A}. If we write $(v_1,v_2,v_3)=\left( \frac{A_k}{n}, \frac{B_k}{n}, \frac{C_k}{n} \right)$, this is bounded by
\begin{equation}\label{eqn_bound_derivative_psi}
\frac{1}{\epsb_{k+1}^{3/2} \sqrt{n}} \sum_{i=1}^3 |v_i| \times \max_{\substack{|x-\mathscr{X}(k/n)|\leq |v_1| \\ |y-\mathscr{Y}(k/n)|\leq |v_2| \\ |z-\mathscr{Z}(k/n)|\leq |v_3|}} \left| \frac{\partial \psi_X}{\partial x}(x,y,z) \right|.
\end{equation}
Just like for $\phi_X$ in the proof of Proposition~\ref{prop_drift_estimates}, we can compute the gradient of $\psi_X$: the partial derivatives are of the form $\frac{P(x,y,z)}{(x+y+z)^4}$, where $P$ is a homogeneous polynomial of degree $3$. By using Lemma~\ref{lem_XYZ_not_too_small}, just like in~\eqref{eqn_bd_second_derivative_A}, we have that $x$, $y$ and $z$ are respectively $O \left( \epsb_k^2 n^{1/100} \right)$, $O \left( \epsb_k n^{1/100} \right)$ and $O \left( \epsb_k^{3/2} n^{1/100} \right)$ and that the sum $x+y+z$ is of order $\epsb_k$. Therefore, by considering the higher order terms in the polynomial $P(x,y,z)$, we obtain the following estimates:
\[ \frac{\partial \psi_X}{\partial x}(x,y,z) = O \left( \epsb_k^{-1} n^{3/100} \right), \quad \frac{\partial \psi_X}{\partial y}(x,y,z) = O \left( \epsb_k^{-1/2} n^{3/100} \right), \quad \frac{\partial \psi_X}{\partial z}(x,y,z) = O \left( \epsb_k^{-1} n^{3/100} \right).\]
Combining this with~\eqref{eqn_bound_derivative_psi}, we get that the third term of~\eqref{eqn_decomposition_variance_A} is
\[ O \left( \frac{n^{3/100} \log n}{\epsb_k^2 n^{3/2}} \right).\]
using $\epsb_k \geq n^{-2/5-1/100}$, this is $O \left( \frac{n^{-1/30}}{\epsb_k n} \right)$.

We now bound the fourth term of~\eqref{eqn_decomposition_variance_A}. For this, we use again~\eqref{eqn_fluidlimit_near_end}. In particular, when we write down $\psi_X \left( \mathscr{X}, \mathscr{Y}, \mathscr{Z}  \right)(t_{\ext}-\eps)$, the highest order terms in $\eps$ are $\frac{\mathscr{Y}\mathscr{Z}}{\mathscr{S}^2} \sim \sqrt{3} \sqrt{\eps}$ and $\frac{\mathscr{Y}^2\mathscr{Z}}{\mathscr{S}^3} \sim \sqrt{3} \sqrt{\eps}$, so we have
\[ \psi_X \left( \mathscr{X}, \mathscr{Y}, \mathscr{Z}  \right)(t_{\ext}-\eps) \sim_{\eps \to 0} 2\sqrt{3} \sqrt{\eps}.\]
In particular, taking $\eps=\epsb_k$, if we choose $\eta$ small enough the fourth term of~\eqref{eqn_decomposition_variance_A} is bounded by $\frac{\delta}{\epsb_k n}$. Finally the fifth term is also smaller than $\frac{\delta}{\epsb_k n}$ if $ \epsb_{k}$ is small enough. Gathering-up the pieces we have established \eqref{eqn_variance_estimate_A}.

The bound~\eqref{eqn_variance_estimate_A_bis} follows from the same proof by noticing that the only term of \eqref{eqn_decomposition_variance_A} which makes the errors $\frac{\widetilde{A}^2}{\epsb_k^{1/2} n}, \frac{\widetilde{B}^2}{n}, \frac{\widetilde{C}^2}{n}$ appear is the $-\E \left[ \Delta X_k | \mathcal{F}_k \right]^2$, which is negative.

\noindent \textsc{Variance of $\widetilde{B}$}. The bound~\eqref{eqn_variance_estimate_B} is immediate: for the same reason as with $\widetilde{A}$, we have
\[\var \left( \Delta \widetilde{B}_k | \mathcal{F}_k \right) = \frac{1}{n} \var \left( \Delta Y_k | \mathcal{F}_k \right) \leq \frac{9}{n},\]
since $|\Delta Y_k|$ is bounded by $3$.

\noindent \textsc{Variance of $\widetilde{C}$}. Finally, we prove~\eqref{eqn_variance_estimate_C}: as before, we can write
\begin{align*}
\var \left( \Delta \widetilde{C}_k | \mathcal{F}_k \right) &= \frac{1}{\epsb_{k+1} n} \E \left[ (\Delta Z_k)^2 | \mathcal{F}_k \right] - \E \left[ \Delta Z_k | \mathcal{F}_k \right]^2\\
&\leq \frac{1}{\epsb_{k+1} n} \left| \E \left[ (\Delta Z_k)^2 | \mathcal{F}_k \right] - \Psi_Z \left( \frac{X_k}{n}, \frac{Y_k}{n}, \frac{Z_k}{n}\right) \right|
+ \frac{1}{\epsb_{k+1} n} \Psi_Z \left( \frac{X_k}{n}, \frac{Y_k}{n}, \frac{Z_k}{n}\right).
\end{align*}
By the exact same argument as for $\widetilde{A}$, the first term is
\[O \left( \frac{1}{\epsb_{k+1} n} \frac{1}{S_k} \right)=O \left( \frac{1}{\epsb_k^2 n^2} \right) = O \left( \frac{\epsb_k^{1/2}}{\epsb_k n} \right),\]
where the first equality comes from Lemma~\eqref{lem_XYZ_not_too_small} and the second from $\epsb_k \geq n^{-2/5-1/100}$. On the other hand, noticing that every term in $\psi_Z(x,y,z)$ has a factor $\frac{z}{s}$, we can write
\[ \psi_Z \left( \frac{X_k}{n}, \frac{Y_k}{n}, \frac{Z_k}{n} \right) = O \left( \frac{Z_k}{S_k} \right) = O \left( \frac{\epsb_k^{3/2} n}{\epsb_k n} \right) = O \left( \epsb_k^{1/2} \right),\]
where the second inequality comes from Lemma~\ref{lem_XYZ_not_too_small}. This proves~\eqref{eqn_variance_estimate_C}.

\end{proof}

\subsection{ Rough behaviour of $ \widetilde{A}, \widetilde{B}$ and $ \widetilde{C}$}
In this section we will use our drift and variance estimates to control $ \widetilde{A}, \widetilde{B}, \widetilde{C}$. We shall get a rather rough control on $ \widetilde{A}, \widetilde{B}$ and $ \widetilde{C}$ (Proposition \ref{prop:goodregion}) and later refine the one on $ \widetilde{A}$. In the rest of this subsection, on top of the constant $K>0$ given by Propositions \ref{prop_drift_estimates} and \ref{prop_variance_estimates}, we fix $$\delta = \frac{1}{100}$$ for definiteness and let  $ 0< \eta\equiv \eta(\delta) <1/2$ so that we can apply the above propositions. In particular, the value of $\eta$ does not depend on $n$, nor on the coming $ \epsilon>0$ and \emph{its value may be decreased for convenience} by keeping the same $\delta$. In the coming pages $ \mathrm{Cst}>0$ is a constant (which may depend on the constant $K$ or the now-fixed $\delta= \frac{1}{100}$) and  that may increase from line to line, but whose value does not depend upon $n$ (provided it is large enough), nor $ \eta$, nor on the forthcoming $ \epsilon$. On the contrary $K_{ \epsilon}$ is a constant that depends upon $ \epsilon$ but also upon $\eta$ in an implicit way. 

\rev{Recall the notation $\epsb_k$ from~\eqref{eq:notationreste} and the notation $\widetilde{\theta}$ from~\eqref{eq:defthetatilde}}.
The value $\eta$ shall give our ``starting scale" $k_{0}= \lfloor (t_{ \mathrm{ext}}-\eta)n \rfloor $ which is such that $ \epsb_{k_0} = \eta$ and we shall then look at times $ k_0 \leq k \leq \tilde{\theta}$.   
We start by controlling the fluctuations at $k_{0}$. 

\begin{lemma}[Fluctuations in the bulk] \label{lem:enterregion}
For all $ \epsilon >0$ there exists $K_ { \epsilon}>0$ so that for all $n$ large enough, with probability at least $ 1 - \epsilon$ we have
\begin{equation}\label{eq:initialisationk0} \max (| \widetilde{A}_{k_0}|, |\widetilde{B}_{k_0}|,| \widetilde{C}_{k_0}|) < K_ \epsilon \mbox{ and } \widetilde{\theta} > k_0 . \end{equation} 
\end{lemma}

\proof Classical results entail that on top of the law of large numbers for the process $ n^{-1}\cdot ({X}^n, {Y}^n, {Z}^n)$ given in Proposition \ref{prop:fluid-limit}, we have a functional central limit theorem for their fluctuations, as long as we stay in the bulk. More precisely, for $ 0 \leq t \leq (t_ \mathrm{ext}- \eta)$, the solution given by the differential equation \eqref{eqn_diff_system_primo} is bounded away from $0$, i.e.   \begin{eqnarray} \label{eq:inf>0}\inf \{\min( \mathscr{X}(t), \mathscr{Y}(t), \mathscr{Z}(t)) :   0 \leq t \leq (t_ \mathrm{ext}- \eta)\} >0, \end{eqnarray} and thanks to our hypothesis \eqref{eq:convcrit}, the initial fluctuations $A_0$, $B_0$ and $C_0$ are bounded so that $(A_0, B_0,C_0)/ \sqrt{n}$ converges to $(0,0,0)$\footnote{We could have allowed $o( \sqrt{n})$ fluctuations, but not $o(n)$ as in Theorem \ref{thm:phasetransition}.}. 
Therefore, we can apply \cite[Theorem 2.3 p 458]{ethier2009markov}, which implies that
$$ \left(\left(\frac{A_{\lfloor t n\rfloor}}{ \sqrt{n}}, \frac{B_{\lfloor t n\rfloor}}{ \sqrt{n}},\frac{C_{\lfloor t n\rfloor}}{ \sqrt{n}}\right): 0 \leq t \leq t_ \mathrm{ext} - \eta\right) $$ converges as $n$ goes to infinity weakly to a continuous random processes driven by a nice stochastic differential equation. Furthermore \cite[Theorem 2.3 p 458]{ethier2009markov} entails that the terminal value $$\left(\frac{A_{\lfloor (t_ \mathrm{ext}- \eta) n\rfloor}}{ \sqrt{n}}, \frac{B_{\lfloor (t_ \mathrm{ext}- \eta) n\rfloor}}{ \sqrt{n}},\frac{C_{\lfloor (t_ \mathrm{ext}- \eta) n\rfloor}}{ \sqrt{n}}\right)$$ converges towards a Gaussian law whose covariance depends on $\eta$ only. Given \eqref{eq:inf>0}, this implies that w.h.p.~we have $X_{k}>0$ for all $0 \leq k \leq ( t_{ \mathrm{ext}}-\eta)n$ (in other words $\theta > (t_ \mathrm{ext} - \eta) n$) and that $ |\widetilde{A}_{\lfloor (t_ \mathrm{ext}- \eta) n\rfloor}|,|\widetilde{B}_{\lfloor (t_ \mathrm{ext}- \eta) n\rfloor}|, |\widetilde{C}_{\lfloor (t_ \mathrm{ext}- \eta) n\rfloor}|$ are tight. The statement of the lemma follows.  \endproof

After this initial control, we shall provide a rough upper bound on the fluctuation processes. 
\begin{proposition}[Rough upper bounds]\label{prop:goodregion} For all $ \epsilon >0$,  there exists a constant $ K_{ \epsilon} > 0$  such that for $n$ large enough, with probability at least $ 1 - \epsilon$ we have 
\begin{align}
\max_{k_{0} \leq k < \tilde{\theta}}\left\{\frac{|  \widetilde{A}_{ k}|}{|\log ( \epsb_k)|^{{3/4}}}, | \widetilde{B}_{ k}|,| \widetilde{C}_{k}| \right\} & \leq K_{ \epsilon}.
\end{align}
\end{proposition}
\begin{remark}[The truth] The proof of the proposition shows that we can replace the $3/4$ exponent by $1/2 +\delta$ for all $\delta >0$.  We anyway expect an ``iterated logarithm'' behavior for $ \widetilde{A}$ so that we could replace $|\log ( \epsb_{k})|$ by $| \log  \log  (\epsb_{k})|$. In the same vein, a little more effort would  yield that $ \widetilde{B}$ and $ \widetilde{C }$ ``converge"\footnote{To be precise, and stressing the dependence in $n$,  the processes $ (\widetilde{B}^n_{[tn] \wedge \theta^n} : t \in [0, t_{ \mathrm{ext}}])$ converge in law for the $\|\|_\infty$ distance towards a limiting process $ (\mathcal{B}_t : t \in [0, t_{ \mathrm{ext}}])$ which is continuous and in particular continuous at $ t_{ \ext}$. Similarly $ (\widetilde{C}^n_{[tn] \wedge \theta^n} : t \in [0, t_{ \mathrm{ext}}]) \to (\mathcal{C}_t : t \in [0, t_{ \mathrm{ext}}])$ for a random continuous process and furthermore $ \mathcal{C}_{t_{ \mathrm{ext}}} = \frac{3 \sqrt{3}}{2} \mathcal{B}_{t_ \mathrm{ext}}$.} but our estimates will be largely sufficient for our purposes. 
\end{remark}

\proof In light of the form of the drift of $\widetilde{C}$ obtained in Equation \eqref{eqn_drift_estimate_C}, we will rather consider the process $ \widetilde{E}_k = \widetilde{C}_k - \frac{3 \sqrt{3}}{2} \widetilde{B}_k$ instead of $ \widetilde{C}$, but notice we can control $ | \widetilde{C}_{k}| \leq \frac{3 \sqrt{3}}{2} | \widetilde{B}_{k}| + | \widetilde{E}_{k}|$ using the processes $ \widetilde{B}$ and $ \widetilde{E}$ so that it is sufficient to prove the proposition after replacing $ \widetilde{C}$ by $ \widetilde{E}$. Introduce $L$ the first time at which one of the those three processes becomes large, i.e.\ 
\begin{align*}
 L &= \tilde{\theta} \wedge \min\left\{k \geq k_0 :  \max\left( \frac{| \widetilde{A_{k}}|}{|\log  \epsb_{k}|^{3/4}} , |\widetilde{B}_{k}|,|\widetilde{E}_{k}|\right) > K_{ \epsilon}  \right\}.
 \end{align*}
We call the region defined by the above inequalities on $( \widetilde{A}, \widetilde{B}, \widetilde{C})$ the \emph{good region} for the processes and evaluate separately the probability that we exit this region (i.e. that $L < \tilde{\theta}$) via one of the three processes $ \widetilde{A}, \widetilde{B}$ or $\widetilde{E}$. By definition  \eqref{eq:defthetatilde} of  $ \tilde{\theta}$ and since we will always take $n$ large enough to have
$$K_{\epsilon}  (1+\log_2^{3/4} (n))< \log(n),$$ 
 as long as $k_0 \leq k < L$, we can apply the estimates obtained in Propositions \ref{prop_drift_estimates} and \ref{prop_variance_estimates}.  Specifically, we will  decompose the processes $ \widetilde{A}, \widetilde{B}$ and $ \widetilde{E}$ into their predictable and martingale parts and use Doob's maximal inequality and $L^2$ estimates to control the martingales. 
 \\ 
\textsc{Let us start with $ \widetilde{B}$}. 
We write for $ k_0 \leq k \leq L $, 
$$ \widetilde{B}_{ k} =  \widetilde{B}_{ k_0} + \sum_{ \ell = k_0}^{ k -1} \E \left[ \Delta \widetilde{B}_\ell | \mathcal{F}_\ell \right] + M_{ k}^{B}$$ where $(M_{ k \wedge L }^{B})_{ k \geq k_0}$ is an $ (\mathcal{F}_k)$-martingale which starts from $0$ at time $k_{0}$. We first evaluate the drift/predictable part. To ease the calculation and readability, we will deliberately drop the integer-part notation $\lfloor \cdot \rfloor $ and introduce scales. Recall that the value of $\eta = \epsb_{k_0}$ has been fixed above, but we may decrease it for convenience as long as it is independent of $ n$ and $ \epsilon$. We start from $k_0 = (t_{ \mathrm{ext}}- \eta)n$ and  we let $$k_j =  (t_{\mathrm{ext}}- \eta 2^{-j})n,$$ for $ 0 \leq j \leq (\frac{2}{5} + \frac{1}{100}) \log_2(n)$. In particular we have $  j + |\log_{2} \eta| \leq |\log_{2} \epsb_{k}| \leq (j+1) + |\log_{2}\eta|$ for all $k_{j} \leq  k \leq k_{j+1}$. With this notation, and using our estimate \eqref{eqn_drift_estimate_B}, we know that if $k \geq k_0$ we have
\begin{eqnarray*} 
\mathbbm{1}_{k \leq L} \left| \sum_{ \ell = k_0}^{ k -1} \E \left[ \Delta \widetilde{B}_\ell | \mathcal{F}_\ell \right]\right|  
&\underset{\eqref{eqn_drift_estimate_B}}{\leq} & \mathrm{Cst} \sum_{ \ell = k_0}^{ \infty} \mathbbm{1}_{ \ell < L} \left( \frac{ \max \left( | \widetilde{A}_\ell |, |\widetilde{B}_\ell |, |\widetilde{C}_\ell |\right)}{ \sqrt{ \epsb_\ell} n}+ \frac{ 1 }{ \epsb_\ell n} n^{-1/30}\right) \\
&\underset{\text{good region}}{\leq}&  \mathrm{Cst} \cdot K_{ \epsilon} \cdot \sum_{ \ell =k_{0}}^{\infty}\mathbbm{1}_{ \ell < L} \left( \frac{|\log \epsb_{\ell}|^{3/4} +1}{ \sqrt{  \epsb_{\ell}} n} + \frac{1}{  \epsb_{\ell}n } n^{-1/30}\right)\\
 &\underset{\text{scales}}{\leq}&  \mathrm{Cst} \cdot K_{ \epsilon} \cdot \sum_{j=1}^{ \log_{2}(n)} \sum_{\ell = k_{j-1}}^{k_j - 1} \left( \frac{(j + |\log_{2} \eta|)^{3/4}}{ \sqrt{ \eta 2^{-j}}n} + \frac{1}{ \eta 2^{-j} n } n^{-1/30}\right) \\
  &{\leq}&  \mathrm{Cst} \cdot K_{ \epsilon} \cdot \sum_{j=1}^{ \log_{2}(n)}  \left( \sqrt{\eta} \frac{(j + |\log_{2} \eta|)^{3/4}}{ 2^{{j/2}}} +  n^{-1/30}\right) \\
    &{\leq}&  \mathrm{Cst} \cdot K_{ \epsilon}  \cdot  |\log \eta | \cdot \sum_{j=1}^{ \log_{2}(n)}  \left( \sqrt{\eta} \frac{j + 1}{ 2^{{j/2}}} +  n^{-1/30}\right) \\ & \leq&  K_{ \epsilon} \cdot \left( \mathrm{Cst} \cdot   \sqrt{\eta} |\log \eta |\right),
 \end{eqnarray*} 
for $n$ large enough where $\mathrm{Cst}>0$ is a constant that may vary from line to line but that does not depend on $n$, nor on $\epsilon$ nor on $ \eta$ as long as it is small. In particular, we may decrease the value of $\eta$ so that the parenthesis in the last display is smaller than $1/4$ say. We obtain that the sum of the absolute values of the expected conditional drifts of $ \widetilde{B}$ between $k_0$ and $L $ is bounded by $ K_ \epsilon /4$ (deterministicaly).

We deduce that the event $\{ L < \tilde{\theta}  \mbox{ and } | B_{k_0} | \leq K_ \epsilon /4 \mbox{ and } |\widetilde{B}_{ L } |> K_ \epsilon \}$ is included in the event $\{ L < \tilde{\theta} \mbox{ and } | M_{ L }^B |> K_ \epsilon/2\}$ so that in particular we can write 
\begin{align*} \mathbb{P}\left( L < \tilde{\theta} \mbox{ and we exit the region by } \widetilde{B} \right) 
&\underset{\text{\color{white} Doob}}{\leq}  \mathbb{P} \left( |\widetilde{B}_{k_0}| > K_ \epsilon /4\right) + \mathbb{P} \left( L   < \tilde{\theta} \mbox{ and } |M_{ L }^B |> K_ \epsilon / 2\right)   \\
&\underset{\text{\color{white} Doob}}{\leq} \mathbb{P} \left( |\widetilde{B}_{k_0}| > K_ \epsilon /4\right) + \mathbb{P} \left( \sup_{k_0 \leq k  < L} |M_{ k }^B| >  K_ \epsilon /2 \right) \\
&\underset{\text{Doob}}{\leq}  \mathbb{P} \left( |\widetilde{B}_{k_0}| > K_ \epsilon /4\right) + 4 \frac{ \mathbb{E}\left[ (M_{L}^B)^2\right]}{ K_ \epsilon^2 /4}.
\end{align*}
Up to increasing $ K_ \epsilon$ we can bound the first term by $ \epsilon$ using Lemma \ref{lem:enterregion}. To  bound the second term, we use our variance estimate \eqref{eqn_variance_estimate_B} which gives in the good region
$$  \mathbb{E}\left[ (\Delta M_k^B)^2 | \mathcal{F}_k,\,  { k \leq L}\right] =  \mathrm{Var} \left( \Delta M_{k}^B | \mathcal{F}_k,\,  { k \leq L}\right)= \mathrm{Var} ( \Delta \widetilde{B}_{k}  | \mathcal{F}_k,\,  { k \leq L}) \underset{ \eqref{eqn_variance_estimate_B}}{\leq} \frac{K}{n}. $$
By the orthogonality of martingale increments in $L^{2}$  we deduce that 
$$ \mathbb{E} [ (M_ L ^B)^2] = \sum_{k = k_0}^{ \infty} \E \left[  \mathbbm{1}_{ k \leq L} \mathbb{E}[(\Delta M_k^B)^2 | \mathcal{F}_k] \right] \leq \frac{ K (t_ \mathrm{ext} n - k_0)}{n} = K \eta.$$
Hence we obtain 
\begin{align*} 
\mathbb{P}(L < \tilde{\theta} \mbox{ and we exit the good region by } \widetilde{B}  ) &\leq  \epsilon + 16 \frac{ \mathbb{E}[ (M_{ L}^B)^2 ]}{K_ \epsilon^2} 
  \leq \epsilon + \frac{16 K \eta}{ K_ \epsilon^2 }.
 \end{align*}
 If $K_ \epsilon$ is large enough, the second term is also less than $ \epsilon$ so that the probability in the left-hand side is small. Conclusion: it is unlikely that we exit first the good region because of the process $ \widetilde{B}$. 
 
\noindent \textsc{Case of $\widetilde{E}$}. The proof is similar, but we shall use more precisely the form of the conditional expected drifts. As before, we write 
$$ \widetilde{E}_{ k} =  \widetilde{E}_{ k_0} + \sum_{ \ell = k_0}^{ k -1} \E \left[ \Delta \widetilde{E}_\ell | \mathcal{F}_\ell \right] + M_{ k}^{E}$$ where $(M_{ k \wedge L}^{E})_{ k \geq k_0}$ is an $ (\mathcal{F}_ k)$-martingale which starts from $0$ at time $k_{0}$. We will bound $ \mathbb{P}( L < \tilde{\theta} \mbox{ and } \widetilde{E}_L > K_ \epsilon )$ and the case $ \widetilde{E}_L <- K_ \epsilon$ will be treated similarly. Let us introduce $ L^{-}_{E}$, the last time before $ L$ where $ \widetilde{E}$ is smaller that $ K_ \epsilon /2$. In particular on the event $\{L < \tilde{\theta} \mbox{ and } \widetilde{E}_L > K_ \epsilon\}$, for $L^{-}_{E} < k \leq L$ the process $ \widetilde{E}$ is larger than $ K_ \epsilon /2$ and its conditional expected drift therefore satisfies   \begin{eqnarray*}\left| \E \left[ \Delta \widetilde{E}_k | \mathcal{F}_k \right] -\frac{1}{\epsb_k n}  \underbrace{\widetilde{E}_{k}}_{\geq K_{\epsilon}/2}\right| &\underset{ \eqref{eqn_drift_estimate_B} \&\eqref{eqn_drift_estimate_C}}{\leq}&  \frac{3 \sqrt{3}}{2} \left(\frac{K}{\epsb_k n} \sqrt{\epsb_k} \max \left( |\widetilde{A}_k|, |\widetilde{B}_k|, |\widetilde{C}_k| \right) + \frac{K}{\epsb_k n} n^{-1/30}\right) \\ 
   & &+ \frac{\delta}{\epsb_k n} \max \left( |\widetilde{B}_k|, |\widetilde{C}_k| \right) + \frac{K}{\epsb_k n} \epsb_k^{3/4} |\widetilde{A}_k| + \frac{K}{\epsb_k n} n^{-1/30} \\
  & \leq &   \frac{(\delta + 3K \sqrt{ \epsb_{k}})}{\epsb_k n} \max \left( |\widetilde{B}_k |, |\widetilde{C}_k | \right) + \frac{4K}{\epsb_k n} \epsb_k^{1/2} |\widetilde{A}_k| + \frac{4K}{\epsb_k n} n^{-1/30}\\
     &\underset{\substack{\text{good region} \\ \text{$n$ large enough}}}{\leq}& K_{ \epsilon} \cdot  \left( \frac{2(\delta + 3K\sqrt{ \epsb_{k}})}{  \epsb_{k} n} + \frac{ 4K \epsb_{k}^{1/2}|\log \epsb_{k}|^{3/4}}{ \epsb_{k} n}\right).  \end{eqnarray*}  Up to further diminishing $\eta$ (which forces $  \epsb_{k} < \eta$ to be small), we can assume that the right-hand side is smaller than $ K_{ \epsilon}/(4 \epsb_{k} n)$ for $n$ large enough so that we are sure that the conditional expected drift $\E \left[ \Delta \widetilde{E}_k | \mathcal{F}_k \right]$ is less than $- K_{ \epsilon}/(4 \epsb_{k} n)$ for $L^{-}_{E} < k <L$ and in particular it is negative and pulls back the process towards $0$.

\begin{figure}[!h]
 \begin{center}
 \includegraphics[width=14cm]{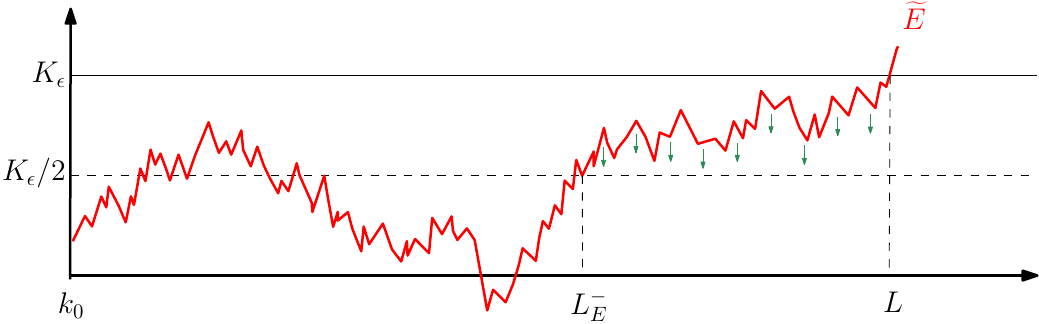}
\caption{Illustration of the proof. If we exit the good region through the process $ \widetilde{E}$, then it has a negative drift (green arrows on the figure) over the time interval $(L^{-}_{E},L)$ and this forces its martingale part to vary too much.}
 \end{center}
 \end{figure}

We deduce that on the event $\{k_0 < L^{-}_{E} < L < \tilde{\theta} \mbox{ and } \widetilde{E}_L > K_ \epsilon\}$ the variation of the martingale $M^E$ over $[L^{-}_{E},L]$ must be larger than $ K_ \epsilon/2$ (just because the drift plays against the process in this region).  Hence, 
\begin{align*} \mathbb{P}( L < \tilde{\theta} \mbox{ and } \widetilde{E}_L > K_ \epsilon ) 
& \leq  \mathbb{P}(L^{-}_{E} \leq k_{0}) + \mathbb{P} \left( \sup_{k_0 \leq k \leq L} |M_{k \wedge L}^E| >  \frac{K_\epsilon}{4}\right)  \end{align*}
We now use our variance estimates \eqref{eqn_variance_estimate_C} and  \eqref{eqn_variance_estimate_B} in the good region. In particular, 
\begin{align*} \mathbb{E}\left[(\Delta M_k^E)^2 \mathbbm{1}_{ k \leq L}\right] &=  \mathrm{Var} \left( \Delta M_{k}^E \mathbbm{1}_{ k \leq L}\right)= \mathrm{Var} \left( \left(\Delta \widetilde{C}_{k} - \frac{3 \sqrt{3}}{2} \Delta \widetilde{B}_{k}\right)  \mathbbm{1}_{ k \leq L}\right) \\
& \leq  \mathrm{Cst} \cdot \left( \mathrm{Var} (\Delta \widetilde{C}_{k} \mathbbm{1}_{ k \leq L})  + \mathrm{Var}(\Delta \widetilde{B}_{k} \mathbbm{1}_{ k \leq L}) \right)   \underset{ \eqref{eqn_variance_estimate_B} \& \eqref{eqn_variance_estimate_C}}{\leq}\frac{ \mathrm{Cst}}{ \sqrt{ \epsb_k} n },
\end{align*}
where $ \mathrm{Cst}>0$ as usual does not depend on $n$ nor on $ \epsilon$ nor on $\eta$.  Summing those variances over one scale  we obtain $$ \sum_{k = k_i}^{k_{i+1} - 1}\frac{ \mathrm{Cst}}{ \sqrt{ \epsb_k} n } \leq \frac{  \mathrm{Cst} (k_{i+1}- k_i ) }{ \sqrt{ \epsb_{k_{i+1}}} n} \leq  \mathrm{Cst}\frac{ \eta 2^{-i} n }{ \sqrt{ \eta 2^{-i}}n}  = \mathrm{Cst}\sqrt{ \eta} (\sqrt{2})^{-i}.$$
We deduce that 
$$
 \mathbb{P} \left( \sup_{ k_0 \leq k \leq  L} |M_{k}^E |> \frac{ K_ \epsilon}{4} \right) 
    \underset{ \mathrm{Doob}}{\leq} \frac{ 16 \mathbb{E} [ (M_{ L}^E)^2]}{K_ \epsilon^2} 
  \leq \frac{ \mathrm{Cst}}{K_ \epsilon^2} \sum_{i = 0}^{ \infty} \sum_{k = k_i}^{k_{i+1} - 1}\mathbb{E}[(\Delta M_k^E)^2 \mathbbm{1}_{k \leq L}] 
  \leq \frac{ \mathrm{Cst} \sqrt{ \eta}}{K_ \epsilon^2}.$$
  If $K_ \epsilon$ is large enough, this bound, as well as  $\mathbb{P}(L^{-}_{E} \leq k_{0})$ (by Lemma~\ref{lem:enterregion}), are less than $ \epsilon$ so the probability of the event $ \{ k_{0}< L < \tilde{\theta} \mbox{ and } \widetilde{E}_L >~K_ \epsilon \} $ is less than $2 \epsilon$. Combined with the symmetric case when $ \widetilde{E}_{L}<-K_{ \epsilon}$, this finishes the case of $\widetilde{E}$.\\
\textsc{Let's finally move on to the control of  $ \widetilde{A}$} .
Again, we decompose $ \widetilde{A}$ as follows
$$ \widetilde{A}_{ k} =  \widetilde{A}_{ k_0} + \sum_{ \ell = k_0}^{ k -1} \E \left[ \Delta \widetilde{A}_\ell | \mathcal{F}_\ell \right] + M_{ k}^{A},$$ where $(M_{ k\wedge L}^{A})_{ k \geq k_0}$ is a martingale for the canonical filtration and starts at $0$ at time $k_{0}$. Compared to the above cases, we shall look more precisely at the scale of $L$ and introduce 
$$ J \mbox{ such that }  k_{J} \leq L < k_{J+1} . $$
In particular, recall that if $k_{j} \leq k \leq k_{j+1}$ we have $j + |\log_{2} \eta| \leq |\log_{2} \epsb_{k}| \leq (j+1) + |\log_{2}\eta|$ so that up to losing a multiplicative factor, we may replace $ |\log  \epsb_{k}|$ by the corresponding scale $j$ in the calculations. As before, let us bound from above the probability that we exit the good region with the process $\widetilde{A}$, that is $$ \P ( L < \tilde{\theta} \mbox{ and }  \widetilde{A}_{ L} > K_ \epsilon \cdot (J+1)^{3/4})$$ and the case $L < \tilde{\theta} \mbox{ and }  \widetilde{A}_{L} <- K_ \epsilon \cdot (J+1)^{3/4}$ is symmetric. As for the case of $ \widetilde{E}$,  we introduce $L_A^-$ the last time before $L$ where $ \widetilde{A}$ is smaller than $ K_ \epsilon (J+1)^{3/4}/2$ and $I$ its corresponding scale (i.e. such that $k_I \leq L_A^- < k_{I+1}$), see Figure \ref{fig:driftA}.  As before, we get from Lemma \ref{lem:enterregion} that $L_A^- >k_{0}$ with high probability when $K_{\epsilon}$ is large. 
We will now use the fact that the conditional expected drift of $\widetilde{A}$ not only goes against $ \widetilde{A}$ but also that its strength is linear in $ \widetilde{A}$. 

\begin{figure}[!h]
 \begin{center}
 \includegraphics[width=15cm]{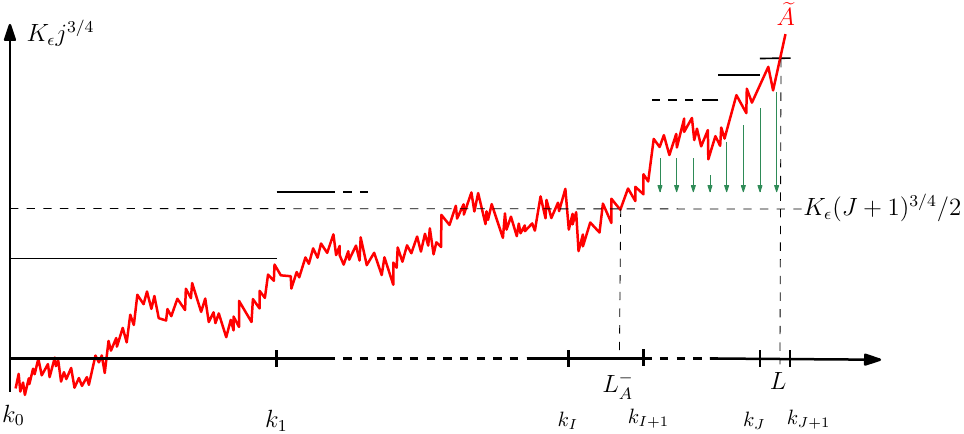}
\caption{Illustration of the proof. If we exit the good region through the process $ \widetilde{A}$, then it has a negative drift (green arrows on the figure) over the time interval $(L_A^-,L)$ whose strength is proportional to $ J^{3/4} K_{ \epsilon}$ over a scale. As in the above cases, this forces its martingale part to vary too much. \label{fig:driftA}}
 \end{center}
 \end{figure}
 
Specifically, when $L_A^- < k < L$, the process $\widetilde{A}$ is larger than $ K_ \epsilon (J+1)^{3/4}/2$ while the other processes are in absolute value less than $ K_{ \epsilon}$ so that by \eqref{eqn_drift_estimate_A} the predictable drift is negative and of order $- \widetilde{A}_k/( \epsb_k n)$: 

 \begin{eqnarray} \left| \E \left[ \Delta \widetilde{A}_k | \mathcal{F}_k \right] -\frac{1}{4 \epsb_k n}  \underbrace{\widetilde{A}_{k}}_{\geq K_ \epsilon (J+1)^{3/4}/2}\right| &\underset{ \eqref{eqn_drift_estimate_A} }{\leq}& \frac{\delta}{\epsb_k n} |\widetilde{A}_k| + \frac{K \epsb_k^{1/4}}{\epsb_k n} \max \left( |\widetilde{B}_k|, |\widetilde{C}_k| \right) + \frac{K}{\epsb_k n} n^{-1/30}\nonumber \\
  &\underset{\text{good region}}{\leq}& \frac{1/10}{\epsb_k n} K_ \epsilon (J+1)^{3/4} + \frac{K \epsb_k^{1/4}}{\epsb_k n} K_{\epsilon} + \frac{K}{\epsb_k n} n^{-1/30} \nonumber \\ &\leq& \frac{1}{9 \epsb_k n} K_ \epsilon (J+1)^{3/4},  \label{eq:calculdriftA}\end{eqnarray}
  for $n$ large enough up to diminishing $\eta$ if necessary. In particular,  $ \E \left[ \Delta \widetilde{A}_k | \mathcal{F}_k \right]$ is less than $- \frac{ K_{ \epsilon}}{100} \frac{(J+1)^{3/4}}{ \epsb_{k}n}$ and summing the conditional expected drift  over all $k \in (L_A^-,L)$ yields total drift smaller than   
\begin{align*} 
 \sum_{k =L_A^- +1}^{ L-1} \mathbb{E}[ \Delta \widetilde{A}_k | \mathcal{F}_k] &\leq  -c \cdot (J-I-1) K_{ \epsilon} (J+1)^{{3/4}},
  \end{align*}
 for some constant $c>0$.  Let us first concentrate on the case where $I+1 < J$ so that $J-I-1>0$. In particular, the variation of the martingale $M^A$ between $ L_A^-+1$ and $L$ must compensate this drift and must be larger than $-c(J-I-1) {K_ \epsilon}(J+1)^{3/4}$. Thus,  we have 
\begin{align*} \mathbb{P}&\left( k_{0}< L_A^- < L < \tilde{\theta} \mbox{ and } I+1 <J  \mbox{ and } \widetilde{A}_{ L}/(J+1)^{3/4} > K_ \epsilon \right) \\
& \leq \sum_{j=2}^{\log_2(n)}\sum_{i=0}^{j-2}  \mathbb{P}\left(\sup_{k_i \leq \ell < k_{j+1} \wedge L} M_\ell^A - \inf_{k_i \leq \ell < k_{j+1} \wedge L} M_\ell^A> c (j-i-1) K_ \epsilon (j+1)^{3/4}\right) \\
& \underset{Doob}{\leq}  \mathrm{Cst}\sum_{j=2}^{\log_2(n)} \sum_{i=0}^{j-2}\frac{ \mathbb{E}[(M_{k_{j+1} \wedge L}^A - M_{k_{i}}^A)^2] }{K_ \epsilon^2 (j+1)^{3/2} (j-i-1)^2}.
\end{align*}
  Thanks to our variance estimates \eqref{eqn_variance_estimate_A_bis} we have $ \mathbb{E}[(\Delta M_k^A)^2 \mathbf{1}_{k <L}] \leq  \frac{ \mathrm{Cst}}{ \epsb_k n}$ so that after summing over scales we obtain  $\mathbb{E}[(M_{k_{j+1} \wedge L}^A - M_{k_{i}}^A)^2]  \leq   \mathrm{Cst} \cdot (j+1-i)$. Plugging this back into the above estimate we deduce 
 \begin{align*} \mathbb{P}&\left( k_{0}< L_A^- < L < \tilde{\theta} \mbox{ and } I+1 <J  \mbox{ and } \widetilde{A}_{ L}/(J+1)^{3/4} > K_ \epsilon \right) \\
 & \leq  \frac{\mathrm{Cst}}{K_\epsilon^2}\sum_{j=2}^{\log_2(n)} \sum_{i=0}^{j-2}\frac{ (j+1-i)}{ (j+1)^{3/2} (j-i-1)^2} \leq \frac{\mathrm{Cst}}{K_ \epsilon^2},
 \end{align*}
so that this probability can be made arbitrarily small by making $ K_{\epsilon}$ large. The case  ${\widetilde{A}_{ L}/(J+1)^{3/4} < -K_ \epsilon}$ is treated similarly. As for the case $|I-J| \leq 1$, since $k_{0}<L_A^-$ w.h.p. (by Lemma~\ref{lem:enterregion}), we use that in this case the martingale $M^A$ must have a variation of at least $ K_{ \epsilon} (j+1)^{3/4}/2$ over $(k_{j-1}, k_{j+1})$ (we do not use the strength of the drift, but just the fact it plays against us over $(L_A^-,L)$ as for the case of $ \widetilde{E}$). By Doob maximal inequality and the above estimate, this probability is bounded from above by $ \mathrm{Cst}/((j+1)^{3/2} K_{\epsilon}^{2})$, whose sum over $0 \leq j  \leq \log_2(n)$ is $\leq \frac{ \mathrm{Cst}}{K_\epsilon^{2}}$. We conclude similarly.
 \endproof

  \subsection{This is the end}
  
  Using Proposition \ref{prop:goodregion} and \eqref{eqn_fluidlimit_near_end}, we can conclude as in Lemma \ref{lem_XYZ_not_too_small} that the process $X_{k}$ stays positive at least as long as 
 $$ n \epsb_{k}^{2} \gg \sqrt{n}  (\epsb_{k})^{3/4} |\log \epsb_{k}| ^{{3/4}}, \quad \mbox{ i.e. as long as } \quad t_{ \mathrm{ext}}n - k \gg n^{3/5} (\log n)^{3/5}.$$
 Through a more refined control on $ \widetilde{A}$, we shall first prove that we can  remove the $\log^{3/5} n$ and prove that $nt_{ \mathrm{ext}} - \theta = O_{ \mathbb{P}}(n^{{3/5}})$, see Proposition \ref{prop:controlA}. The convergence in law of $n^{{-3/5}}(nt_{ \mathrm{ext}} - \theta)$  will be deduced by doing a SDE approximation for the process $\widetilde{A}_{k}$ when $ \epsb_{k} \approx n^{-2/5}$ in Proposition \ref{prop:end}. \\
 
  Since we now take a close look at times $ k = t_{ \mathrm{ext}}n - O (n^{3/5})$, let us introduce a new piece of notation: for $ k \geq 0$, we write 
 $$ \xb_{k} :=  n^{-3/5}(t_{ \mathrm{ext}}n - k), \quad \mbox{so that} \quad k = t_{ \mathrm{ext}}n - \xb_{k} n^{3/5} \quad \mbox{i.e.} \quad  \epsb_{ k} =  \xb_{k} n^{-2/5}.$$
 With this notation at hands, we can state a refined control on $ \widetilde{A}$.

\begin{proposition}[Control on $ \widetilde{A}$ in the critical region]\label{prop:controlA} For all $ \epsilon >0$ there exists $K_{ \epsilon}$ such that with probability at least $ 1 - \epsilon$,  for all $ k \leq t_{ \mathrm{ext}}n$ such that $\xb_{k} \geq K_{ \epsilon}$, we have 
$$ |\widetilde{A}_{k}| < K_{ \epsilon} \xb_{k}^{1/8}.$$
\end{proposition}
In particular, performing the same argument as in the beginning of this subsection, we deduce that $X_{k}$ stays positive until time $t_ \mathrm{ext} n - K_ \epsilon n^{3/5}$, that is $\theta > t_ \mathrm{ext} n - K_ \epsilon n^{3/5}$ with probability at least $1 - \epsilon$.

\proof The proof is similar to the control of $ \widetilde{A}$ in Proposition \ref{prop:goodregion}. With the notation of the proof of Proposition \ref{prop:goodregion}, let us introduce
$$ T =L \wedge \min\{k \geq k_0  :  | \widetilde{A}_ k | > K_ \epsilon  \cdot \xb_{k}^{ \frac{1}{8}}  \}$$ 
and $J \in \{0,1,2,...\}$ the corresponding scale, i.e.\ such that $ 2^{J}   \geq \xb_{T}  >  2^{(J-1)} $. As for the previous control of $ \widetilde{A}$, we will replace $\xb_{T}$ by $ 2^J$ in the calculation to make the reading easier.  Note in particular that $k \mapsto \xb_{k}$ is decreasing. 

We bound the probability $ \P ( T= k < L \mbox{ and }    \widetilde{A}_ k  > K_ \epsilon  \cdot 2^{ \frac{J}{8}} )$, the case $ \{T= k < L \mbox{ and }    \widetilde{A}_ k  <- K_ \epsilon  2^{ \frac{J}{8}}\}$  being similar. For this, let $\alpha>0$ be a small constant (to be precised later), and let
$$ T^- = \sup\left\{ k_{0} \leq k \leq T : \widetilde{A}_{k} \leq  \alpha (I-J +1) K_ \epsilon 2^{J/8} \mbox{ with } 2^{I-1}<  \xb_k \leq 2^{I}\right\}$$ and $I \geq J$ its corresponding scale (notice the slight difference here with the proof of Proposition \ref{prop:goodregion} because $I$ enters in the definition of the barrier). As before, Lemma \ref{lem:enterregion} will entail that $T^{-}>k_0$ with high probability as $n \to \infty$ and when $ k_0 \leq T^- \leq k \leq T$ and $2^{i-1} < \xb_{k} \leq 2^{i}$,  we have $  {\widetilde{A}_{k} \geq   \alpha(i-J +1) K_ \epsilon 2^{J/8}}$. By the same calculation as in \eqref{eq:calculdriftA} we have 
$$  \E \left[ \Delta \widetilde{A}_k | \mathcal{F}_k \right] \leq - \frac{\alpha}{8} \frac{(i-J +1) K_ \epsilon 2^{J/8}}{\epsb_k n}.$$
Summing those expected conditional drifts  over all $T^{-} +1 \leq k < T$ yields a total drift smaller than   
\begin{eqnarray*} 
 \sum_{k = T^{-}+1}^{T - 1}\mathbb{E}[ \Delta \widetilde{A}_k | \mathcal{F}_k] &\underset{\text{scales}}{\leq}&  \sum_{i = J+1}^{I-1} \sum_{k \geq 0} \mathbbm{1}_{2^i \geq \xb_k > 2^{i-1}}\mathbb{E}[ \Delta \widetilde{A}_k | \mathcal{F}_k] \\
 &\leq& -\frac{\alpha}{8} \sum_{i = J+1}^{I-1} (i-J+1) K_ \epsilon 2^{J/8} \sum_{k \geq 0} \mathbbm{1}_{2^i \geq \xb_k > 2^{i-1}} \frac{1}{  \epsb_{k}n} \\
  &\leq& -\frac{\alpha}{16} \sum_{i = J+1}^{I-1} (i-J+1) K_ \epsilon 2^{J/8} \\
    &\leq& -\frac{\alpha}{16} ( I-J-1)^2 K_ \epsilon 2^{J/8}. 
 \end{eqnarray*}
Let us first focus on the case $I-J \geq 2$: as soon as $T^{-}>k_{0}$  the variation of the martingale $M^A$ between $T^{-}$ and $T$ must compensate this drift plus the difference of the starting and ending values, and so must be larger than 
 \[ K_{\epsilon} 2^{J/8} \left( \frac{\alpha}{16}(I-J-1)^2 -\alpha(I-J+1)+2^{-1/8} \right).\]
If $\alpha$ has been chosen small enough (e.g. $\alpha=\frac{1}{100}$), as soon as $I-J \geq 2$, this is larger $\frac{1}{32} K_{\epsilon} 2^{J/8} (I-J)^2$.  As in the proof of Proposition \ref{prop:goodregion}, the sum of the variances of the increments of $M^{A}$ between scales $i$ and $j$ is bounded above by $ \mathrm{Cst}(i-j)$ and so the probability that $M^{A}$ varies by more than $\frac{1}{32}( i-j)^2 K_ \epsilon 2^{j/8}$ over this time interval is bounded above using Doob's inequality by 
 $$  \mathrm{Cst}\frac{i-j}{\left(( i-j)^2 K_ \epsilon 2^{j/8}\right)^{2}}.$$
Summing these probabilities over all scales $j_{0} \leq j \leq i$, we deduce that 
\begin{align*} \mathbb{P}&\left( k_{0}< T^- < T < L \mbox{ and } I -1 >J\geq j_{0}  \mbox{ and } \widetilde{A}_{T} > K_ \epsilon 2^{J/8} \right) \\
& \leq  \frac{\mathrm{Cst}}{K_\epsilon^2}\sum_{i \geq j+2 \geq j_{0}+2}\frac{ i-j}{ (i-j)^4 2^{j/4}} \leq \frac{\mathrm{Cst} \cdot 2^{-j_{0}/4}}{K_ \epsilon^2},
 \end{align*}
and this can be made arbitrarily small by taking $j_{0}$ large enough. Finally, we treat the case $0 \leq I-J \leq 1$ similarly, by noting that in this case, if $k_{0}<T^{-}$ (which has high probability by Lemma \ref{lem:enterregion}), the variation of $\widetilde{A}$ between times $T^-$ and $T$ is at least $(2^{-1/8}-2\alpha) K_{\epsilon} 2^{J/8}$. Since the drift is negative, the martingale $M^A$ must have a variation of order $ K_ \epsilon 2^{J/8}$ (provided $\alpha<\frac{1}{4}$) over the scale $J$, and the conclusion is the same.
\endproof

In the rest of this subsection we stress back the dependence in $n$ and use $\theta^{n} \equiv \theta$ for the stopping time of the exploration and study the convergence of $$ \xb_{ \theta^{n}}  \in \mathbb{R} \quad \mbox{ such that } \quad  \theta^{n} = t_{ \mathrm{ext}}n - \xb_{ \theta^{n}} n^{3/5}.$$
\begin{proposition}\label{prop:end} We have the following convergence in distribution as $n$ goes to infinity
$$ \xb_{ \theta^{n}}  \xrightarrow[n\to\infty]{(d)}  3^{-3/5} \cdot 2^{4/5} \cdot \vartheta^{-2},$$
where $\vartheta = \inf\{ t \geq 0 : W_{t} = -t^{-2}\}$ with $W$ a standard linear Brownian motion started from $0$ at $0$.
\end{proposition}

 \proof  Fix $ \epsilon>0$ and let $K_{ \epsilon}>0$ so that on an event $ \mathcal{E}_{n}$ of probability at least $1 - 3\varepsilon$, the conclusions of Lemma \ref{lem:enterregion}, Proposition \ref{prop:controlA} and Proposition \ref{prop:goodregion} hold. Fix $ K_{ \epsilon}^{-1} > \xi>0$ small enough so that $K_{ \epsilon} \xi^{1/8} \leq  \epsilon$. We shall first focus on the times $k$ satisfying $  \xi \leq  \xb_{k} \leq \xi^{-1}$ and consider the renormalized process $$ \widetilde{F}_{k} = \frac{ \widetilde{A}_{k}}{ \xb_ k^{1/4}}, \quad 0 \leq k \leq \theta^{n}.$$ Let us compute its conditional expected drift and variance: for $k < \tilde{\theta}^{n}$ with $\xi \leq  \xb_{k} \leq \xi^{-1}$, on the event $ \mathcal{E}_{n}$ the assumptions of Proposition \ref{prop_drift_estimates} hold, so that using $ \epsb_{k} = n^{-2/5} \xb_{k}$ we have
\begin{align} 
  \mathbb{E} [ \Delta \widetilde{F}_k | \mathcal{F}_k, \mathcal{E}_{n}] &\leq \frac{\delta}{\xb_k n^{3/5}} |\widetilde{A}_k| + \frac{K}{\xb_k n^{3/5}} n^{-1/30} =  \frac{\delta}{\xb_k^{3/4} n^{3/5}} |\widetilde{F}_k| + \frac{K}{\xb_k n^{3/5}} n^{-1/30}   \label{eq:driftF}\\
\left|\mathrm{Var} \left( \Delta \widetilde{F}_k | \mathcal{F}_k,\mathcal{E}_{n}\right) - \frac{2 \sqrt{3}}{\xb_k^{3/2} n^{3/5}}\right| &= \left|\frac{1}{\xb_k^{1/2}}\mathrm{Var} \left( \Delta \widetilde{A}_k | \mathcal{F}_k\right) -   \frac{2 \sqrt{3}}{\xb_k^{3/2}n^{3/5}}\right| \leq \frac{ \delta}{\xb_k^{3/2} n^{3/5} }.\label{eq:varianceF}
 \end{align} 
 We now make $\delta$ vary with $n$ and take $ \delta \equiv \delta_n \xrightarrow[n\to\infty]{} 0$ in the above displays. Indeed, using the notation of Propositions \ref{prop_drift_estimates} and \ref{prop_variance_estimates} we can do so as soon as $ \eta ( \delta_n) >  1/\xi \cdot n^{-2/5}$. 
 To avoid stopping times issues, we possibly extend $ \widetilde{F}$ after time $ \theta^{n}$ (in the case  $ \xb_{\theta^{n}}  \leq \xi$) by a process $ \widehat{F}$ whose increments are $\pm (\frac{2 \sqrt{3}}{\xb_k^{3/2}n^{3/5}})^{1/2}$ with probability $1/2$ (in particular independent, centered, with variance $\frac{2 \sqrt{3}}{ \xb_k^{3/2}n^{3/5}}$ and whose $L^{ \infty}$-norm tends to $0$ uniformly as $ n \to \infty$), so that our estimates \eqref{eq:driftF} and \eqref{eq:varianceF} remain true for all $ \{ k : \xi \leq  \xb_{k} \leq \xi^{-1}\}$. 
 Let us recapitulate what we have: with probability at least $1- 3 \epsilon$ for all $ \{ k : \xi \leq  \xb_{k} \leq \xi^{-1}\}$:
$$ \left\{\begin{array}{l}
\displaystyle| \widehat{F}_{ nt_{ \mathrm{ext}}- \xi^{{-1}}n^{3/5}}| < \epsilon, \quad (\mbox{by Prop. \ref{prop:controlA} and the assumption }K_{ \epsilon} \xi^{1/8} \leq  \epsilon),
\\
\ \\ 
\displaystyle\mathbb{E} [ \Delta \widehat{F}_k | \mathcal{F}_k] = o(n^{-3/5}) \cdot | \widehat{F}_k| +o(n^{-3/5}), \\
\displaystyle \mathrm{Var} \left( \Delta \widehat{F}_k | \mathcal{F}_k\right) = \frac{2 \sqrt{3}}{\xb_k^{3/2} n^{3/5}} + o(n^{-3/5}),\\
\displaystyle \|\Delta \widehat{F}_k\|_{\infty} = o(1),
\end{array}\right.$$
where the $o(1)$ function is uniform in $\{ k : \xi \leq  \xb_{k} \leq \xi^{-1}\}$. By standard results in diffusion approximation, see e.g.~\cite{kushner1974weak}, this implies the following weak convergence for the $\|\|_{\infty}$-norm:
$$\left(\widehat{F}_{t_ \mathrm{ext} n - t n^{3/5}}-\widehat{F}_{t_ \mathrm{ext} n - \xi^{{-1}}n^{3/5}}\right)_{ \xi\leq t \leq \xi^{-1} } \xrightarrow[n\to\infty]{} ( \mathcal{H}_{t})_{\xi \leq t \leq \xi^{-1}},$$ where the process $\mathcal{H}$ satisfies the stochastic differential equation (in reverse time) $ \mathrm{d} \mathcal{H}_{-t} = \frac{\sqrt{2 \sqrt{3}}}{ t^{3/4}} \mathrm{d}B_{-t}$ with initial condition $ \mathcal{H}_{\xi^{-1}}=0$.  By \rev{Dubins}-Schwarz theorem, the solution of this SDE can be written as  $$2 \cdot 3^{1/4} \left(W_{ \frac{1}{ \sqrt{t}}}- W_{ \frac{1}{ \sqrt{\xi^{-1}}}}\right)_{ \xi \leq t \leq \xi^{-1}}$$ where $W$ is a standard linear Brownian motion with $W_{0}=0$. Letting $ \epsilon \to 0$ and $\xi \to 0$, we deduce the following weak convergence over all compact subsets of $ (0,\infty)$:
  \begin{eqnarray}\left(\widehat{F}_{t_ \mathrm{ext} n - tn^{3/5}}\right)_{ 0 < t < \infty } \xrightarrow[n\to\infty]{} \left(W_{ \frac{1}{ \sqrt{t}}}\right)_{ 0 < t < \infty }.   \label{eq:convBrownien1}\end{eqnarray}

To see that the above convergence implies the convergence of stopping times recall that
 \begin{eqnarray*} \xb_{\theta^{n}} := \sup\{ \xb_{k} \geq 0,\ X_{k}= 0\}&=&\sup\{ \xb_{k} \geq 0,\ \widetilde{F}_k = - n^{4/5} \mathscr{X}(k/n)/ \xb_k\}\\ &=& \sup\{ \xb_{k} \geq 0,\ \widehat{F}_k \leq - n^{4/5} \mathscr{X}(k/n)/\xb_k\}. \end{eqnarray*} 
In particular, the time $ \xb_{\theta^{n}}$ can be seen as the first time when started from $+\infty$ that the process $ \widehat{F}$ crosses the barrier $ \mathscr{C}^n$ defined by $$ \mathscr{C}^n( \xb_{k}) = - n^{4/5} \mathscr{X}(k/n)/\xb_k.$$
Recalling \eqref{eqn_fluidlimit_near_end}, we have $- n^{4/5} \mathscr{X}(k/n)/\xb_k \sim -3 \xb_k$, so that the barrier $ \mathscr{C}^n$ converges towards the graph $ \mathscr{C}$ of the function $ t \mapsto -3t$. Since the crossing of $ \mathcal{C}$ by  $\left(W_{ {1/\sqrt{t}} }: 0 < t < \infty \right)$ when started from $+\infty$ happens at an almost surely positive time $\tau$ and since $W$ immediately takes values strictly above and below $ \mathscr{C}$ after hitting it, it follows that 
$$ \xb_{\theta^{n}} \xrightarrow[n\to\infty]{(d)} \tau = \sup \{ t \geq 0 :\ 2 \cdot 3^{1/4}\cdot W_{ \frac{1}{ \sqrt{t}}} = -3 t\}.$$
By scaling we have the equality in distribution   \begin{eqnarray*} \tau &\overset{(d)}{=}&\sup \{ t \geq 0 :\ 2 \cdot 3^{1/4}\cdot W_{ \frac{1}{ \sqrt{t}}} = -3 t\}\\
&\underset{u = 1/\sqrt{t}}{{=}}& \left( \inf \{ u \geq 0 : W_{u} = \frac{3^{{3/4}}}{2} u^{-2}\}\right)^{-2}\\
&\underset{\alpha>0}{\overset{(d)}{=}}& \left( \inf \{ u \geq 0 : \frac{1}{\sqrt{\alpha}}\cdot W_{\alpha u} = \frac{3^{{3/4}}}{2} u^{-2}\}\right)^{-2}\\
&\underset{\alpha u= v}{{=}}& \left( \frac{1}{\alpha} \inf \{ v \geq 0 :  W_{v} =  \sqrt{\alpha} \alpha^{2} \cdot \frac{3^{{3/4}}}{2} v^{-2}\}\right)^{-2}\\
&\underset{\alpha^{5/2} \cdot \frac{3^{{3/4}}}{2}=1}{{=}}&  \left( \frac{3^{{3/4}}}{2}\right)^{-4/5} \left(\inf \{ v \geq 0 :  W_{v} = v^{-2}\}\right)^{-2}.  \end{eqnarray*} The statement follows. \endproof

\subsection{Proof of Theorem \ref{thm:maincritical}: Size and composition of the KS-Core}
We have now all the ingredients to prove our main Theorem \ref{thm:maincritical}. First by Proposition \ref{prop:end}, the  renormalized ending time $  \xb_{\theta^{n}}$ converges in distribution to $  2^{4/5}3^{-3/5} \vartheta^{-2}$ where $\vartheta$ is the hitting time of the curve $ t \mapsto -t^{-2}$ by a Brownian motion. At this time, by Proposition \ref{prop:goodregion} and \eqref{eqn_fluidlimit_near_end} we have 
$$ Y_{\theta^{n}} \quad = \underbrace{B_{ \theta^{n}}}_{  \underset{ \mathrm{Prop}. \ref{prop:goodregion}}{\leq}  \mathrm{Cst}\sqrt{n} \log(n)^{3/4}} + \underbrace{n \mathscr{Y} \left( \frac{\theta^{n}}{n}\right)}_{ \underset{\eqref{eqn_fluidlimit_near_end}}{\sim} 4 \xb_{\theta^{n}} n^{3/5}} \quad \rev{= o_{ \mathbb{P}}(n^{3/5}) + 2^{14/5}3^{-3/5} \vartheta^{-2} \xb_{\theta^{n}} n^{3/5} },$$
$$ Z_{\theta^{n}} \quad = \underbrace{C_{ \theta^{n}}}_{  \underset{ \mathrm{Prop}. \ref{prop:goodregion}}{\leq} \mathrm{Cst} \,n^{3/10} \log(n)^{3/4}} + \underbrace{n \mathscr{Z} \left( \frac{\theta^{n}}{n}\right)}_{ \underset{\eqref{eqn_fluidlimit_near_end}}{\sim} 4 \sqrt{3} \rev{ \xb_{\theta^{n}}^{3/2}} n^{2/5}} \quad \rev{=o_{ \mathbb{P}}(n^{2/5}) + 2^{16/5} 3^{-2/5} \vartheta^{-3} \xb_{\theta^{n}}^{3/2} n^{2/5} }.$$

Moreover using Proposition \ref{prop:markovexplo}, the KS-Core is just obtained by pairing the remaining half-edges uniformly at random. Our theorem follows. \rev{\sout{Ouff.}}

 \section{Comments} \label{sec:comments}
 We conclude this paper with a few perspectives that our work opens.
 \paragraph{Near critical heuristics.}\label{subsec:near_critical} \rev{The exact same proof would work if the initial degree distribution is critical in the sense of Theorem~\ref{thm:phasetransition}. Moreover, our proof still works as long as the initial fluctuations are $O(n^{1/2})$ (beyond that, it is not possible anymore to use directly the Ethier--Kurtz results on the bulk of the exploration in Lemma~\ref{lem:enterregion}).} However, we believe that our techniques can be used to tackle the near-critical window for the Karp-Sipser core. In particular, this window should be obtained by starting from 
\begin{eqnarray*}d_{1,c}^{n} = n (1-\frac{ \sqrt{3}}{2}) + O(n^{3/5}), \quad 2 d_{2,c}^{n} = O(n^{3/5}) , \quad \mbox{ and } 3 d_{3,c}^{n} = n \frac{ \sqrt{3}}{2} +O(n^{3/5}),  \end{eqnarray*} whereas we studied only the critical case \eqref{eq:convcrit}. All these shifts in the starting configuration should result in a shift of order $ O(n^{3/5})$ of the absorption time. In a similar vein, one could study the ``Phase 2'' of the Karp-Sipser algorithm \cite{aronson1998maximum} which, in the supercritical case, consists in removing a uniform vertex when there are no leaves left. The analysis of this phase should be intimately connected to the above near-critical dynamics.

\paragraph{Universality.}
Obviously, we conjecture that the geometry of the critical core and the scaling limits results are independent of the fine details of the model of random graph we started with. In particular, it should hold for the Erd{\H{o}}s-R\'enyi case or for configuration models with small enough degrees. However, proving a general result seems challenging because we heavily rely on the exact form of the fluid-limit of our exploration processes (such results are available for the Erd{\H{o}}s--R\'enyi case, see \cite{aronson1998maximum}).


\paragraph{Comparison with the $k$-core phase transition.} \label{sec:kcore}
Finally, it is interesting to compare our results with the appearance of the $k$-core in random graphs as studied in \cite{pittel1996sudden,janson2007simple}, where the phase transition is discontinuous \rev{for $k \geq 3$}.

Recall that the $k$-core of a graph $ \mathfrak{g}$ is the maximal subgraph of $ \mathfrak{g}' \subset \mathfrak{g}$ so that the induced degree inside $ \mathfrak{g}'$ of each of its vertices is at least $k$. The emergence of a giant $k$-core has been studied for the Erd{\H{o}}s--R\'enyi random graph and the configuration model, see \cite{pittel1996sudden,janson2007simple}. A difference with the Karp--Sipser core  is that the phase transition is discontinuous: when the $k$-core exists asymptotically, its proportion is bounded away from $0$. This can be explained heuristically as follows.

Suppose for the discussion that $k=3$ and that we are interested in the size of the $3$-core in a configuration model on vertices of degrees $1,2,3$ and $4$. As in the case of the Karp--Sipser algorithm, one can reveal the $3$-core by iteratively taking a leg attached to a vertex of degree $\leq 2$, remove it, and destroy the vertex it is attached to as well as the connection it makes in the graph (hence diminishing the unmatched degree of the vertices in question).  As in this paper, if one starts with some proportions $p_1,p_2,p_3,p_4$ of legs attached to vertices of degree one, two, three and four, we can write the differential equation governing the fluid limit of this process, see \cite{janson2007simple}. The main difference with the Karp--Sipser core is that in this case, the number of legs attached to leaves (to be precise to vertices of degree $1$ or $2$) is not necessarily decreasing. Actually, in the critical case, the fluid limit of the proportion of vertices of degrees $1,2$ follows a curve which is tangent to the boundary of the domain at some point before diving back into the bulk of the simplexe and dying at the right corner, see Figure \ref{fig:kcore} (and compare with Figure \ref{fig:equadiff_trajectories}). This explains the first-order phase transition in this case: a slight perturbation of the initial conditions may push the curve to exit the domain at a very different location. 

\begin{figure}[!h]
 \begin{center}
 \includegraphics[width=10cm]{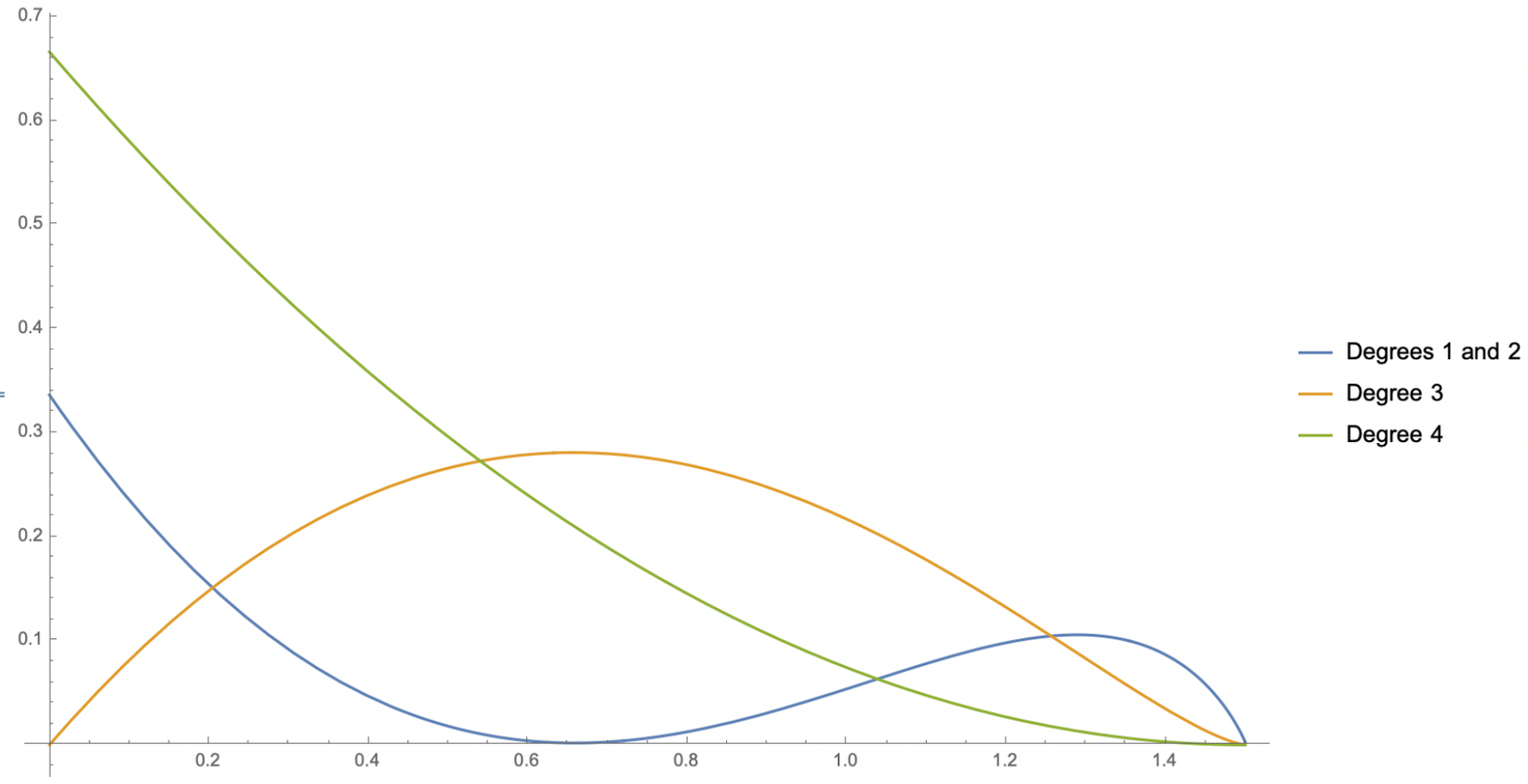} \hspace{0.5cm}\includegraphics[width=5cm]{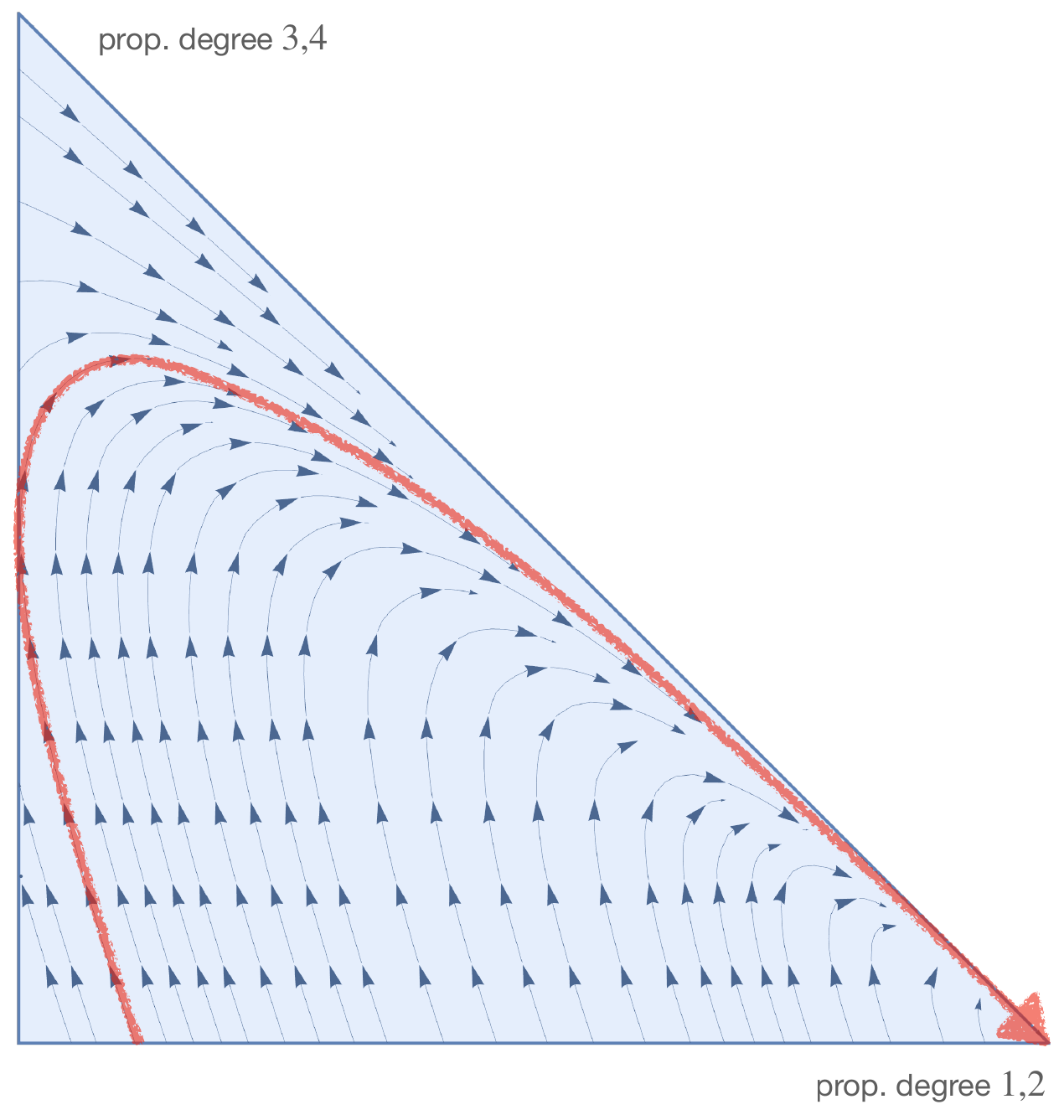}
 \caption{Illustration of the fluid limit of the renormalized number of legs attached to vertices of degree $4,3$ and $ \leq 2$ in the $k$-core algorithm at the critical point. In particular, a slight perturbation of the initial conditions may cause a drastic change of the absorption time of the system and this explains why the $k$-core percolation exhibits a first-order phase transition. \label{fig:kcore}}
 \end{center}
 \end{figure}

\bibliographystyle{siam}
\bibliography{bibli}

\end{document}